\documentclass{amsart}
\usepackage{amsmath,amssymb}
\usepackage{enumerate,color}

\numberwithin{equation}{section}

\newtheorem{theorem}{Theorem}[section]

\newtheorem{corollary}[theorem]{Corollary}
\newtheorem{definition}[theorem]{Definition}
\newtheorem{lemma}[theorem]{Lemma}
\newtheorem{remark}[theorem]{Remark}

\newenvironment{dedication}
        {\vspace{6ex}\begin{quotation}\begin{center}\begin{em}}
        {\par\end{em}\end{center}\end{quotation}}

\begin{document}

\title[\emph{}]{Trace theorem for quasi-Fuchsian groups}

\author[]{A. Connes}
\address{College de France, 3 rue d'Ulm, Paris F-75005 France}
\email{alain@connes.org}

\author[]{F. Sukochev}
\address{School of Mathematics and Statistics, University of NSW, Sydney,  2052, Australia}
\email{f.sukochev@unsw.edu.au}


\author[]{D. Zanin}
\address{School of Mathematics and Statistics, University of NSW, Sydney,  2052, Australia}
\email{d.zanin@unsw.edu.au}

\maketitle

\begin{dedication} Dedicated to Dennis Sullivan.
\end{dedication}
\begin{abstract} We complete the proof of the  Trace Theorem in the quantized calculus for quasi-Fuchsian group which was stated and sketched, but not fully proved, on pp. 322-325 in the book \lq\lq Noncommutative Geometry\rq\rq of the first author.
\end{abstract}

\section{Introduction}

We first recall how quasi-Fuchsian groups are obtained by Bers  (\cite{BE}) from a pair of cocompact Fuchsian groups $\Gamma_1, \Gamma_2$  and a given group isomorphism $\alpha:\Gamma_1\to \Gamma_2$. All required notations and notions used below are explained in Section \ref{prel section}. The quasi-Fuchsian group $G=G(\Gamma_1, \Gamma_2,\alpha)$ is a discrete subgroup $G\subset {\rm PSL}(2,\mathbb{C})$ which simultaneously uniformizes the compact Riemann surfaces $X_j=\mathbb{D}/\Gamma_j,$ $j=1,2,$ (where $\mathbb{D}$ is the unit disk in $\mathbb{C}$)  in the following sense (\cite{BO}):
\begin{enumerate}
\item There is a Jordan curve  $C\subset\bar{\mathbb{C}}=S^2$ invariant under any $g\in G$ and such that the action of $G$ on $C$ is minimal (every orbit is dense).
\item Let $\Sigma_{{\rm int}}$ and $\Sigma_{{\rm ext}}$ be the connected components of the complement of $C$. There are conformal diffeomorphisms $Z:\mathbb{D}\to\Sigma_{{\rm int}}$, $Z':\mathbb{D}\to\Sigma_{{\rm ext}}$ and group isomorphisms $\pi:G\to \Gamma_1$, $\pi':G\to \Gamma_2$ such that 
$$g\circ Z=Z\circ\pi(g),\quad g\circ Z'=Z'\circ\pi'(g),\quad \pi'(g)=\alpha(\pi(g))\quad  \forall g\in G.$$	
\end{enumerate}

 Furthermore, the group $G=G(\Gamma_1,\Gamma_2,\alpha)$ satisfies the following properties:
\begin{enumerate}[{\rm (i)}]
\item $G$ is finitely generated.
\item $G$ does not contain elliptic or parabolic elements.
\end{enumerate}

The Jordan curve $C=\Lambda(G)$ is a quasi-circle whose Hausdorff dimension $p$ is strictly bigger than one except when the $\Gamma_1$ and $\Gamma_2$ are conjugate Fuchsian groups (\cite{BO}, Theorem 2).

The main result of this paper is the following theorem appearing as Theorem 17 on p. 324 in \cite{Connes}. It gives a formula for the $p-$dimensional geometric\footnote{A measure $\nu$ on $\bar{\mathbb{C}}$ is called $p-$dimensional geometric (relative to $G$) if $d(\nu\circ g)(z)=|g'|^p(z)d\nu(z)$ for every $g\in G.$ Here, $g'$ is the complex derivative.} probability measure on $C=\Lambda(G)$ in terms of the quantized differential $[F,Z]$ of the Riemann mapping $Z:\mathbb{D}\to\Sigma_{{\rm int}}$  understood as a function on the circle $\mathbb{S}^1=\partial\mathbb{D}$ (to which it extends by  continuity using the Caratheodory theorem (\cite{markushevich})). Here $F$ is the Hilbert transform on the circle; equivalently, $F=2P-1,$ where $P$ is the Riesz projection and the algebra $L_{\infty}(\partial\mathbb{D})$ is identified with its natural action on the Hilbert space $L_2(\partial\mathbb{D})$ by pointwise multiplication. The basic formula depends on the fact that, unlike for distributional derivatives, one can take the $p$-th power $|[F,Z]|^p$ of the absolute value of the quantized differential $[F,Z]$.  The nice geometric properties of the quasi-Fuchsian groups $G=G(\Gamma_1,\Gamma_2,\alpha)$ are used crucially in the proof and we formulate our result in a slightly greater generality and  in more intrinsic terms without reference to the joint uniformization.

\begin{theorem}\label{main theorem}  Let $G$ be a finitely generated quasi-Fuchsian group without  parabolic elements. Let $p>1$ be the Hausdorff dimension of $C=\Lambda(G)$,  and let $\nu$ be the (unique) $p-$dimensional geometric probability measure on $\Lambda(G).$ Then
\begin{enumerate}[{\rm (a)}]
\item\label{mainta} $[F,Z]\in\mathcal{L}_{p,\infty}.$
\item\label{maintb} for every $f\in C(\Lambda(G))$ and for every bounded trace\footnote{in particular for every Dixmier trace} $\varphi$ on $\mathcal{L}_{1,\infty},$ there exists a constant $c(G,\varphi)<\infty$ such that
\begin{equation}\label{trace formula}
\varphi((f\circ Z)\cdot|[F,Z]|^p)=c(G,\varphi)\cdot\int_{\Lambda(G)}f(t)d\nu(t).
\end{equation}
\item\label{maintc} for any  Dixmier trace ${\rm Tr}_{\omega}$, with $\omega$ power invariant, one has $c(G,{\rm Tr}_{\omega})>0.$
\end{enumerate}
\end{theorem}

The statement \eqref{maintc} provides a large class of traces for which $c(G,\varphi)>0.$ The notion of power invariance for the limiting process $\omega$ is explained in Section \ref{app1}.

Theorem \ref{main theorem} was stated in \cite{Connes} and the proof\footnote{which was joint work with D. Sullivan to whom the first author is indebted for his generosity in sharing his geometric insight.} was sketched there after the statement of the Theorem and using a number of lemmas but the reference [538] was never published and the detailed proof is thus unpublished even if the various steps were described in \cite{Connes}. It is thus very valuable to make them available while proving a more general result and introducing variants in the proposed proof in \cite{Connes}. The variants concern the estimate of the growth of the Poincar\'e series which in \cite{Connes} is attributed to Corollary 10 of \cite{sullivan} but the precise relation with the two forms of the absolute Poincare series is assumed without a precise reference. This relation is due to the convex co-compactness of the action of the quasi Fuchsian group inside hyperbolic three space, but in this paper the same estimate is obtained using a different method. The other important point not contained in \cite{Connes} is the proof of the Lemma $3.\beta.11,$ which is stated there without proof.

We are  grateful to our colleagues Christopher Bishop, Magnus Goffeng, Denis Potapov and Caroline Series for their help in the preparation of this paper.

\section{Preliminaries}\label{prel section}

\subsection{General notation}\label{general}
Fix throughout a separable infinite dimensional Hilbert space $H.$ We let $\mathcal{L}(H)$ denote the $*-$algebra of all bounded operators on $H.$ It becomes a $C^*-$algebra when equipped with the uniform operator norm (denoted here by $\|\cdot\|_{\infty}$). For a compact operator $T$ on $H,$ let $\lambda(k,T)$ and $\mu(k,T)$ denote its $k$-th eigenvalue and $k$-th largest singular value (these are the eigenvalues of $|T|$ arranged in the descending order). The 
sequence $\mu(T)=\{\mu(k,T)\}_{k\geq0}$ is referred to as  the singular value sequence of the operator $T.$ The standard trace on $\mathcal{L}(H)$ is denoted by ${\rm Tr}.$ For an arbitrary operator $0\leq T\in\mathcal{L}(H),$ we set
$$n_T(t):={\rm Tr}(E_T(t,\infty)),\quad t>0,$$
where 
$E_T{(a,b)}$ stands for the  spectral projection of a self-adjoint operator $T$ corresponding to the interval $(a,b).$ Fix an orthonormal basis in $H$ (the particular choice of basis is inessential). We identify the  algebra $l_{\infty}$ of bounded sequences with the subalgebra of all diagonal operators with respect to the chosen basis. For a given sequence $\alpha\in l_{\infty},$ we denote the corresponding diagonal operator by ${\rm diag}(\alpha).$

Similarly, let $(X,\kappa)$ be a measure space (finite or infinite, atomless or atomic). For a measurable function $x$ on $(X,\kappa),$ we write
$$n_{|x|}(t)=\kappa(\{u:\ |x|(u)>t\}),\quad \mu(s,x)=\inf\{t:\ n_{|x|}(t)>s\}.$$

\subsection{Principal ideals $\mathcal{L}_{p,\infty}$ and infinitesimals of order $\frac1p$}\label{pinfinity}

For a given $0<p<\infty,$ we let $\mathcal{L}_{p,\infty}$ denote the principal ideal in $\mathcal{L}(H)$ generated by the operator ${\rm diag}(\{(k+1)^{-1/p}\}_{k\geq0}).$ Equivalently,
$$\mathcal{L}_{p,\infty}=\{T\in\mathcal{L}(H): \mu(k,T)=O((k+1)^{-1/p})\}.$$
These ideals, for different $p,$ all admit an equivalent description in terms of spectral projections, namely
\begin{equation}\label{tlp}
T\in\mathcal{L}_{p,\infty}\Longleftrightarrow n_{|T|}(\frac1n)=O(n^p).
\end{equation}
We also have
\begin{equation}\label{115}
|T|^p\in\mathcal{L}_{1,\infty}\Longleftrightarrow \mu^p(k,T)=O((k+1)^{-1})\Longleftrightarrow T\in\mathcal{L}_{p,\infty}.
\end{equation}
The ideal $\mathcal{L}_{p,\infty},$ $0<p<\infty,$ is equipped with a natural quasi-norm\footnote{A quasinorm satisfies the norm axioms, except that the triangle inequality is replaced by $||x+y||\leq K(||x||+||y||)$ for some uniform constant $K>1$.}
$$\|T\|_{p,\infty}=\sup_{k\geq0}(k+1)^{1/p}\mu(k,T),\quad T\in\mathcal{L}_{p,\infty}.$$
However, for $1<p<\infty,$ it is technically convenient to use an equivalent norm
$$\|T\|_{p,\infty}=\sup_{n\geq0}(n+1)^{\frac1p-1}\sum_{k=0}^n\mu(k,T),\quad T\in\mathcal{L}_{p,\infty}.$$

The following H\"older property (see \cite{BS4} Section 6 of Chapter 11) is widely used throughout the paper:
\begin{equation}\label{lpi mult}
A_k\in\mathcal{L}_{p_m,\infty},\ 1\leq m\leq n,\Longrightarrow \prod_{m=1}^nA_m\in\mathcal{L}_{p,\infty},\ \frac1p=\sum_{m=1}^n\frac1{p_m}.
\end{equation}

Similarly, let $(X,\kappa)$ be a measure space (finite or infinite, atomless or atomic). We define a function space
$$L_{p,\infty}(X,\kappa)=\{x\mbox{ is $\kappa-$measurable }:\ \mu(t,x)=O(t^{-\frac1p})\}.$$

In \cite{Connes}, a compact operator $T\in\mathcal{L}(H)$ is called an infinitesimal. It is said to be of order $\alpha>0$ if it belongs to the ideal $\mathcal{L}_{\frac1{\alpha},\infty}.$ Equation \eqref{lpi mult} manifests the fundamental fact that the order of the product of infinitesimals is the sum of their orders.

\subsection{Traces on $\mathcal{L}_{1,\infty}$}

\begin{definition}\label{trace def} If $\mathcal{I}$ is an ideal in $\mathcal{L}(H),$ then a unitarily invariant linear functional $\varphi:\mathcal{I}\to\mathbb{C}$ is said to be a trace.
\end{definition}

Since $U^{-1}TU-T=[U^{-1},TU]$ for all $T\in\mathcal{I}$ and for all unitaries $U\in\mathcal{L}(H),$ and since the unitaries span $\mathcal{L}(H),$ it follows that traces are precisely the linear functionals on $\mathcal{I}$ satisfying the condition
$$\varphi(TS)=\varphi(ST),\quad T\in\mathcal{I}, S\in\mathcal{L}(H).$$
The latter may be reinterpreted as the vanishing of the linear functional $\varphi$ on the commutator 
subspace which is denoted $[\mathcal{I},\mathcal{L}(H)]$ and defined to be the linear span of all commutators $[T,S]:\ T\in\mathcal{I},$ $S\in\mathcal{L}(H).$ 

It is shown in \cite[Lemma 5.2.2]{LSZ} that $\varphi(T_1)=\varphi(T_2)$ whenever $0\leq T_1,T_2\in\mathcal{I}$ are such that the singular value sequences $\mu(T_1)$ and $\mu(T_2)$ coincide. For $p>1,$ the ideal $\mathcal{L}_{p,\infty}$ does not admit a non-zero trace while for $p=1,$ there exists a plethora of traces on $\mathcal{L}_{1,\infty}$ (see e.g. \cite{DFWW} or \cite{LSZ}). An example of a trace on $\mathcal{L}_{1,\infty}$ is the Dixmier trace introduced in \cite{Dixmier} that we now explain.

\begin{definition} The dilation semigroup on $L_{\infty}(0,\infty)$ is defined by setting
$$(\sigma_sx)(t)=x(\frac{t}{s}),\quad t,s>0.$$
In this paper {\it a dilation invariant extended limit} means a state on the algebra $L_{\infty}(0,\infty)$ invariant under $\sigma_s,$ $s>0,$ which vanishes on every function with bounded support.
\end{definition}

\noindent{\bf Dixmier trace}. Let $\omega$ be a dilation invariant extended limit. Then the functional ${\rm Tr}_{\omega}:\mathcal{L}_{1,\infty}^+\to\mathbb{C}$ defined by setting\footnote{Here, singular value function is defined by the formula $\mu(A)=\sum_{k\geq0}\mu(k,A)\chi_{(k,k+1)}.$}
$${\rm Tr}_{\omega}(A)=\omega\Big(t\to\frac1{\log(1+t)}\int_0^t\mu(u,A)du\Big),\quad 0\leq A\in\mathcal{L}_{1,\infty},$$
is additive and, therefore, extends to a trace on $\mathcal{L}_{1,\infty}.$ We call such traces  {\it Dixmier traces}.

These traces clearly depend on the choice of the functional $\omega$ on $L_{\infty}(0,\infty).$ Using a slightly different definition, this notion of trace was applied in \cite{Connes} in the setting of noncommutative geometry. We also remark that the assumption used by Dixmier of translation invariance for the functional $\omega$ is redundant (see \cite[Section IV.2.$\beta$]{Connes} or \cite[Theorem 6.3.6]{LSZ}).

An extensive discussion of traces, and more recent developments in the theory, may be found in \cite{LSZ} including a discussion of the following facts.
\begin{enumerate}[{\rm (a)}]
\item All Dixmier traces on $\mathcal{L}_{1,\infty}$ are positive.
\item All positive traces on $\mathcal{L}_{1,\infty}$ are continuous in the quasi-norm topology.
\item There exist positive traces on $\mathcal{L}_{1,\infty}$ which are not Dixmier traces  (see \cite{SSUZ-pietsch}).
\item There exist traces on $\mathcal{L}_{1,\infty}$ which fail to be continuous (see \cite{DFWW}).
\end{enumerate}

\subsection{Kleinian groups}

A Fuchsian (resp.~Kleinian) group is Poincar\' e's name for a discrete subgroup of ${\rm PSL}(2,\mathbb{R})$ (resp.~of ${\rm PSL}(2,\mathbb{C})$). We are interested in Kleinian groups which are obtained by deforming certain Fuchsian groups. A nice deformation of a Fuchsian group uniformizing a compact Riemann surface is called by Bers a quasi-Fuchsian group (\cite{BE}). The corresponding action on the complex sphere $\bar{\mathbb{C}}$ is topologically conjugate to the action of the Fuchsian group and Poincar\' e noticed the  deformation of the round circle of the Fuchsian group into a topological Jordan curve with remarkable properties. 
This ``so called curve" in the words of  Poincar\' e is now understood to have very nice conformally self-similar properties. We give below the formal definitions (\ref{defn kleinian}, \ref{fuchsian}) of Kleinian, Fuchsian and quasi-Fuchsian groups and work with intrinsic properties of the Kleinian  groups with no mention of the deformation.

We let ${\rm SL}(2,\mathbb{C})$ be the group of all $2\times 2$ complex matrices with determinant $1.$ We identify the group ${\rm PSL}(2,\mathbb{C})={\rm SL}(2,\mathbb{C})/\{\pm 1\}$ and its action on the complex sphere $\bar{\mathbb{C}}$ (see \cite{markushevich}) by fractional linear transformations. The element
$$g=
\begin{pmatrix}
g_{11}&g_{12}\\
g_{21}&g_{22}
\end{pmatrix}\in{\rm SL}(2,\mathbb{C})
\mbox{ represents the mapping }
z\to\frac{g_{11}z+g_{12}}{g_{21}z+g_{22}},\quad z\in\bar{\mathbb{C}}.
$$

The following definition of a Kleinian group is taken from \cite{maskit} II.A. We refer the reader to \cite{maskit} for more advanced properties of Kleinian groups.

\begin{definition}\label{defn kleinian} Let $G\subset {\rm PSL}(2,\mathbb{C})$ be a discrete subgroup. We say that
\begin{enumerate}[{\rm (a)}]
\item $G$ is freely discontinuous at the point $z\in\bar{\mathbb{C}}$ if there exists a neighborhood $U\ni z$ such that $g(U)\cap U=\varnothing$ for every $1\neq g\in G.$
\item $G$ is Kleinian if it is freely discontinuous at some point $z\in\bar{\mathbb{C}}.$
\end{enumerate}
\end{definition}

The set of all points $z\in\bar{\mathbb{C}}$ at which $G$ is {\it not} freely discontinuous is called the limit set of $G$ and is denoted by $\Lambda(G).$ This set is either infinite or consists of $0,$ $1$ or $2$ points. The latter $3$ cases correspond to the so-called {\it elementary} Kleinian groups, which are usually dropped from the consideration.

The definition below can be found in \cite{maskit} on p. 103 and p. 192, respectively\footnote{More precisely what we call \lq\lq quasi-Fuchsian\rq\rq corresponds to \lq\lq quasi-Fuchsian of the first kind\rq\rq}.

\begin{definition}\label{fuchsian} A Kleinian group $G$ is called
\begin{enumerate}[{\rm (a)}]
\item Fuchsian (of the first kind) if its limit set is a circle.
\item quasi-Fuchsian if its limit set is a closed Jordan curve.
\end{enumerate}
\end{definition}

It is known that a limit set of a finitely generated quasi-Fuchsian group (which is not Fuchsian) has Hausdorff dimension strictly greater than $1$ (see Corollary 1.7 in \cite{BJ}).

It is known that $(\bar{\mathbb{C}}\backslash\Lambda(G))/G$ is a Riemann surface for an arbitrary Kleinian group $G.$ The following definition is taken from \cite{BJ}.

\begin{definition} A Kleinian group $G$ is called analytically finite if its Riemann surface $(\bar{\mathbb{C}}\backslash\Lambda(G))/G$ is of finite type; i.e., a finite union of compact surfaces with at most finitely many punctures and branch points.
\end{definition}

We need the important notion of a $p-$dimensional geometric measure on $\bar{\mathbb{C}}.$

\begin{definition} Let $G$ be a Kleinian group. The measure $\nu$ on $\bar{\mathbb{C}}$ is called $p-$dimen\-sional geometric (relative to $G$) if $d(\nu\circ g)(z)=|g'|^p(z)d\nu(z)$ for every $g\in G.$
\end{definition}

An important condition for existence and uniqueness of geometric measures can be found in \cite{sullivan84} (see Theorem 1 there). Our proof of Theorem \ref{main theorem} \eqref{maintb} also delivers, via the Riesz Representation Theorem, the existence of a $p-$dimensional geometric measure concentrated on $\Lambda(G)$ (for the case when $p$ is the Hausdorff dimension of $\Lambda(G)$).

A subgroup in $G$ is called parabolic if it fixes exactly one point in $\bar{\mathbb{C}}.$

The notion of a fundamental domain $\mathbb{F}\subset\bar{\mathbb{C}}$ of a Kleinian group $G$ is defined in \cite{maskit}, II.G. In particular, the sets $\{g\mathbb{F}\}_{g\in G},$ are pairwise disjoint.

We also need the notion of the Hausdorff dimension of a set $X\subset\mathbb{C}$ (applied to the set $\Lambda(G)$ in this text).

\begin{definition} We say that the Hausdorff dimension of a set $X\subset\mathbb{C}$ does not exceed $q$ if there exist balls $B(a_i,r_i)$ such that
$$X\subset\cup_i B(a_i,r_i),\quad \sum_i r_i^q<\infty.$$
The infimum of all such $q$ is called the Hausdorff dimension of a set $X\subset\mathbb{C}.$
\end{definition}

\begin{remark}{\rm In what follows, we may assume without loss of generality that our group $G$ does not contain elliptic elements.  By Selberg's Lemma, there is a torsion-free subgroup $G_0\subset G$ which has finite index in $G.$ The limit set of $G_0$ is the limit set of $G$. Since every finite index subgroup in a finitely generated group is itself finitely generated (see p.~55 in \cite{Rose}), it follows that the conditions of Theorem \ref{main theorem} hold for the group $G_0.$ The proof of this theorem constructs a geometric measure for the subgroup of ${\rm PSL}(2,\mathbb{C})$ of invariance of the limit set of $G_0$ and hence  for the group $G.$ Moreover the uniqueness of the geometric measure for $G_0$ implies uniqueness for $G$. In addition to that, the group $G_0$ does not contain elliptic elements. Indeed, an elliptic element is conjugate in ${\rm PSL}(2,\mathbb{C})$ to a unitary element. Since $G_0$ is discrete, it follows that every elliptic element has finite order; since $G_0$ is torsion free, it follows that there are no elliptic elements.

This remark was written for the reason that some authors do not allow branches in the Riemann surfaces. It is sometimes hard to check whether a particular paper in the reference allows branches or not. The  Riemann surface of a Kleinian group without elliptic elements does not have branches, which makes it easier for the reader. 
}
\end{remark}
\subsection{Action of ${\rm PSL}(2,\mathbb{C})$ on hyperbolic space}

Let us briefly recall how the group ${\rm PSL}(2,\mathbb{C})$ acts on the three dimensional hyperbolic space. We refer the reader to Section 1.2 in \cite{elstrodt} for details.

By definition, the unit ball model $\mathbb{B}$ of hyperbolic space  is the open unit ball of $\mathbb{R}^3$ equipped with the following Riemannian metric.
$$ds^2=4\frac{(du_0)^2+(du_1)^2+(du_2)^2}{(1-u_0^2-u_1^2-u_2^2)^2},\quad u=(u_0,u_1,u_2)\in\mathbb{B}.$$
The Riemannian metric generates a distance in $\mathbb{B}.$ We do not need the (complicated) distance formula, but only the fact that (see formula (2.5) on p.~10 in \cite{elstrodt})
\begin{equation}\label{dist from 0}
{\rm dist}(u,\mathbf{0})=\log(\frac{1+|u|}{1-|u|}),\quad u\in\mathbb{B}.
\end{equation}
Here, $u=(u_0,u_1,u_2)$ is identified with the quaternion $u_0+u_1i+u_2j$ and $|u|$ denotes the norm of the quaternion (which  coincides with the Euclidean norm of $u$).

For a matrix $g\in {\rm SL}(2,\mathbb{C}),$ consider the matrix $\pi(g)$ of quaternions defined as follows
\begin{equation}\label{sl2 rep h3}
\pi(g)=\frac12
\begin{pmatrix}
1&-j\\
-j&1
\end{pmatrix}
g
\begin{pmatrix}
1&j\\
j&1
\end{pmatrix}=
\begin{pmatrix}
a&c'\\
c&a'
\end{pmatrix},\quad |a|^2-|c|^2=1.
\end{equation}
Here, the quaternions $a$ and $c$ are given by the following formulae.
\begin{equation}\label{ac def}
a=\frac12(g_{11}+\bar{g}_{22})+\frac12(g_{12}-\bar{g}_{21})j,\quad c=\frac12(g_{21}+\bar{g}_{12})+\frac12(g_{22}-\bar{g}_{11})j.
\end{equation}
Note that $|a|^2-|c|^2=1.$ The operation $a\to a'$ is the inner automorphism implemented by the quaternion $k,$ it acts as follows
$$(a_0+a_1i+a_2j+a_3k)'=a_0-a_1i-a_2j+a_3k, \ \forall a_j\in \mathbb R.$$
The action of the group ${\rm SL}(2,\mathbb{C})$ on $\mathbb{B}$ is given by the formula
\begin{equation}\label{hyp action}
\pi(g):u\to(au+c')(cu+a')^{-1},\quad u\in\mathbb{B}.\
\end{equation}
By Proposition 1.2.3 in \cite{elstrodt}, this action consists of isometries of $\mathbb{B}.$ Formulae \eqref{dist from 0}, \eqref{ac def} and \eqref{hyp action} are crucially used in the proof of Lemma \ref{reduction to sullivan} below.

\subsection{Bochner integration} The following definition of measurability can be found e.g. in \cite{hf} (see Definition 3.5.4 there).

\begin{definition}\label{bweak} Let $X$ be a Banach space. A function $f:(-\infty,\infty)\to X$ is called
\begin{enumerate}[{\rm (a)}]
\item strongly measurable if there exists a sequence of $X$-valued simple functions converging to $f$ almost everywhere.
\item weakly measurable if the mapping $s\to \langle f(s),y\rangle$ is measurable for every $y\in X^*.$
\end{enumerate} 
\end{definition}

If the Banach space $X$ is separable, then the Pettis Measurability Theorem (see e.g. Theorem 3.5.3 in \cite{hf}) states the equivalence of the notions above.

A strongly measurable function $f$ is Bochner integrable if
\begin{equation}\label{bochner integrability}
\int_{-\infty}^{\infty}\|f(s)\|_Xds<\infty.
\end{equation}
Theorem 3.7.4 in \cite{hf} states that there exists a sequence $\{f_n\}_{n\geq0}$ of simple $X$-valued functions such that
$$\int_{-\infty}^{\infty}\|(f_n-f)(s)\|_Xds\to0,\quad n\to\infty.$$
The Bochner integral is now defined as
$$\int_{-\infty}^{\infty}f(s)ds\stackrel{def}{=}\lim_{n\to\infty}\int_{-\infty}^{\infty}f_n(s)ds.$$
Its key feature is that
$$\Big\|\int_{-\infty}^{\infty}f(s)ds\Big\|_X\leq\int_{-\infty}^{\infty}\|f(s)\|_Xds.$$

\subsection{Weak integration in $\mathcal{L}(H)$}\label{weak int} The following definitions (and subsequent construction of a weak integral) are folklore. For example, one can look at p.~77 in \cite{rudin} and put the topological space $X$ there to be $\mathcal{L}(H)$ equipped with the strong operator topology. Every functional on $X$ can be written as a linear combination of $x\to\langle x\xi,\eta\rangle,$ $\xi,\eta\in H.$

\begin{definition} A function $s\to f(s)$ with values in $\mathcal{L}(H)$ is measurable in the weak operator topology if, for every vectors $\xi,\eta\in H,$ the function
$$s\to\langle f(s)\xi,\eta\rangle,\quad s\in\mathbb{R},$$
is measurable.
\end{definition}

For such functions, there is notion of weak integral. Note that the scalar-valued mapping
$$s\to\sup_{\|\xi\|,\|\eta\|\leq1}\langle f(s)\xi,\eta\rangle=\|f(s)\|_{\infty},\quad s\in\mathbb{R},$$
is measurable.

Let the function $f:\mathbb{R}\to\mathcal{L}(H)$ be measurable in the weak operator topology. We say that $f$ is integrable in the weak operator topology if
\begin{equation}\label{necessary-condition}
\int_{\mathbb{R}}\|f(s)\|_{\infty}ds<\infty.
\end{equation}

Define a sesquilinear form
$$(\xi,\eta)\to\int_{\mathbb{R}}\langle f(s)\xi,\eta\rangle ds,\quad\xi,\eta\in H.$$
It is immediate that
$$|(\xi,\eta)|\leq\int_{\mathbb{R}}\|f(s)\|_{\infty}ds\cdot\|\xi\|\|\eta\|,\quad\xi,\eta\in H.$$
That is, for a fixed $\xi\in H,$ the mapping $\eta\to(\xi,\eta)$ defines a bounded anti-linear functional on $H.$ It follows from the Riesz Lemma (description of the dual of a Hilbert space) that there exists an element $x_{\xi}\in H$ such that $(\xi,\eta)=\langle x_{\xi},\eta\rangle.$ The mapping $\xi\to x_{\xi}$ is linear and bounded. The operator which maps $\xi$ to $x_{\xi}$ is called the {\it weak integral} of the mapping $s\to f(s),$ $s\in\mathbb{R}.$

The so-defined weak integral satisfies the following properties.
\begin{enumerate}[{\rm (a)}]
\item If the mapping $s\to f(s)$ is integrable in the weak operator topology, then
$$\Big\|\int_{-\infty}^{\infty}f(s)ds\Big\|_{\infty}\leq\int_{-\infty}^{\infty}\|f(s)\|_{\infty}ds.$$
\item If the mapping $s\to f(s)$ is integrable in the weak operator topology and if $A\in\mathcal{L}(H),$ then $s\to A\cdot f(s)$ is also integrable in the weak operator topology and
$$\int_{\mathbb{R}}A\cdot f(s)ds=A\cdot\int_{\mathbb{R}}f(s)ds.$$
\item If the mapping $s\to f(s)$ is Bochner integrable in some Banach ideal in $\mathcal{L}(H),$ then it is integrable in the weak operator topology. Its Bochner integral then equals to the weak one.
\end{enumerate}

\subsection{Double operator integrals}\label{doi} Here, we state the definition and basic properties of Double 0perator Integrals which were developed by Birman and Solomyak in \cite{BS1,BS2,BS3}. We refer the reader to \cite{PS-crelle} for the proofs and for more advanced properties.

Heuristically, the double operator integral $T^{X,Y}_{\phi}$, where $X$ and $Y$ are self-adjoint operators and $\phi$ is a bounded Borel measurable function on ${\rm Spec}(X)\times {\rm Spec}(Y),$ is defined using the spectral decompositions:
$$T_{\phi}^{X,Y}(A)=\iint \phi(\lambda,\mu)dE_X(\lambda)AdE_Y(\mu).$$
This formula defines a bounded operator from $\mathcal{L}_2$ to $\mathcal{L}_2.$ However, we want to consider it as a bounded operator on other ideals --- and this leads to  difficulty unless the function $\phi$ is good enough.

To specify the class of \lq\lq good\rq\rq~functions, we use the integral tensor product of \cite{Peller}, of $L^\infty({\rm Spec}(X),\mu_X)$ by $L^\infty({\rm Spec}(Y),\mu_Y)$  where the $\mu$'s denote the spectral measures.  The integral projective tensor products were introduced in \cite{Peller} where it was proved that the maximal class of functions for which the double operator integrals can be defined for arbitrary bounded linear operators coincides with the integral projective tensor product of $L^\infty({\rm Spec}(X),\mu_X)$ by $L^\infty({\rm Spec}(Y),\mu_Y).$ Thus, we consider only those functions $\phi$ which admit a representation
\begin{equation}\label{integral tensor product}
\phi(\lambda,\mu)=\int_{\Omega}a(\lambda,s)b(\mu,s)d\kappa(s),
\end{equation}
where $(\Omega,\kappa)$ is a measure space and where 
\begin{equation}\label{119}
\int_{\Omega}\sup_{\lambda\in {\rm Spec}(X)}|a(\lambda,s)|\cdot\sup_{\mu\in{\rm Spec}(Y)}|b(\mu,s)|d\kappa(s)<\infty.
\end{equation}
 For those functions, we write
\begin{equation}\label{123}
T_{\phi}^{X,Y}(A)=\int_{\Omega}a(X,s)Ab(Y,s)d\kappa(s),
\end{equation}
where the latter integral is understood in the weak sense (the integrand is measurable in the weak operator topology and the condition \eqref{necessary-condition} holds thanks to \eqref{119}).

For the function $\phi$ from the integral tensor product, we have (see Theorem 4 in \cite{PS-crelle}) that $T^{X,Y}_{\phi}:\mathcal{L}_1\to\mathcal{L}_1$ and $T^{X,Y}_{\phi}:\mathcal{L}_{\infty}\to\mathcal{L}_{\infty}.$ In particular, we have that $T^{X,Y}_{\phi}:\mathcal{L}_{p,\infty}\to\mathcal{L}_{p,\infty}$ for $p>1.$

One of the key properties of Double Operator Integrals is that they respect algebraic operations (see e.g. Proposition 2.8 in \cite{DSW} or formula (1.6) in \cite{BS}). Namely,
\begin{equation}\label{doi algebraic}
T_{\phi_1+\phi_2}^{X,Y}=T_{\phi_1}^{X,Y}+T_{\phi_2}^{X,Y},\quad T_{\phi_1\cdot \phi_2}^{X,Y}=T_{\phi_1}^{X,Y}\circ T_{\phi_2}^{X,Y}.
\end{equation}

\subsection{Fredholm modules}

The following is taken from \cite{Connes}.

\begin{definition} Let $\mathcal{A}$ be a $*-$algebra represented on the Hilbert space $H.$ Let $F\in\mathcal{L}(H)$ be self-adjoint unitary operator. We call a triple $(F,H,\mathcal{A})$ Fredholm module if $[F,a]$ is compact for every $a\in\mathcal{A}.$
\end{definition}

The infinitesimal $[F,a]$ is called the quantum derivative of the element $a$ (see Chapter IV in \cite{Connes} for the studies of quantum derivatives).

A Fredholm module is called $(p,\infty)-$summable if $[F,a]\in\mathcal{L}_{p,\infty}$ for every $a\in\mathcal{A}.$

Part \eqref{mainta} of Theorem \ref{main theorem} exactly states that the Fredholm module $(F,L_2(\mathbb{S}^1),\mathcal{A})$ is $(p,\infty)-$summable, where $\mathcal{A}$ is the $*-$algebra generated by $Z.$

\section{Proof of Theorem \ref{main theorem} \eqref{mainta}}

\subsection{Growth of matrix coefficients in $G$} Let $G$ be a Kleinian group. As stated in Corollary II.B.7 in \cite{maskit}, the series $\sum_{g\in G}|g'(z)|^2$ converges for a.e. $z\in\bar{\mathbb{C}}$ (with respect to the Lebesgue measure). The critical exponent of $G$ is defined\footnote{Sullivan uses a slightly different definition in \cite{sullivan}, but they are equivalent.} (see e.g. p. 323 in \cite{Connes}) as follows
$$p=\inf\{q:\ \sum_{g\in G}|g'(z)|^q<\infty\mbox{ for a.e. }z\in\bar{\mathbb{C}}\}.$$

Let $\|g\|_{\infty}$ denote the uniform norm of the matrix $g\in{\rm SL}(2,\mathbb{C})$ as an operator on the Hilbert space $\mathbb{C}^2.$ Equip our countable group $G$ with counting measure and define $l_{p,\infty}(G)$ as in Subsection \ref{pinfinity}.

\begin{lemma}\label{reduction to sullivan} Let $G\subset {\rm PSL}(2,\mathbb{C})$ be a Kleinian group. If $p$ is its critical exponent, then $\{\|g\|_{\infty}^{-2}\}_{g\in G}\in l_{p,\infty}(G).$
\end{lemma}
\begin{proof} By Corollary 5 in \cite{sullivan} (see also the right hand side estimate in Corollary 10 in \cite{sullivan}), we have
$${\rm Card}(\{g\in G:\ {\rm dist}((\pi(g))(\mathbf{0}),\mathbf{0})\leq r\})\leq Ce^{pr}.$$
Using the formula \eqref{dist from 0} and denoting $e^{-r}$  by $t,$ we arrive at
$${\rm Card}(\{g\in G:\ \frac{1-|(\pi(g))(\mathbf{0})|}{1+|(\pi(g))(\mathbf{0})|}\geq t\})\leq Ct^{-p}.$$
Since $|(\pi(g))(\mathbf{0})|<1,$ it follows that
$${\rm Card}(\{g\in G:\ 1-|(\pi(g))(\mathbf{0})|^2\geq 4t\})\leq Ct^{-p}.$$
Since $|a'|=|a|$ and $|c'|=|c|,$ it follows from \eqref{hyp action} that
$$(\pi(g))(\mathbf{0})=c'(a')^{-1}\mbox{ and, therefore, }1-|(\pi(g))(\mathbf{0})|^2=1-\frac{|c|^2}{|a|^2}=\frac1{|a|^2}.$$
Thus,
$${\rm Card}(\{g\in  G:\ \frac1{|a|^2}\geq 4t\})\leq Ct^{-p}.$$
It is immediate from \eqref{ac def} that $|a|\leq 2\|g\|_{\infty}.$ Therefore,
$${\rm Card}(\{g\in G:\ \frac1{4\|g\|_{\infty}^2}\geq 4t\})\leq Ct^{-p}.$$
This concludes the proof.
\end{proof}

By Theorem II.B.5 in \cite{maskit}, $g_{21}\neq0$ for every $1\neq g\in G.$ This allows us to state a stronger version of Lemma \ref{reduction to sullivan}.

\begin{lemma}\label{main kleinian lemma} Let $G$ be a Kleinian group and let $p$ be the critical exponent of $ G.$ If $\infty$ is not in the limit set of $G,$ then $\{|g_{21}|^{-2}\}_{1\neq g\in G}\in l_{p,\infty}(G).$
\end{lemma}
\begin{proof} By the assumption, $\infty\notin\Lambda(G).$ Hence, $G$ is freely discontinuous at $\infty.$ It follows that $\{g(\infty)\}_{1\neq g\in G}$ is a bounded set. Note that $g(\infty)=\frac{g_{11}}{g_{21}}.$ Thus, $|g_{11}|=O(|g_{21}|).$

Clearly,
$$g^{-1}=
\begin{pmatrix}
g_{22}&-g_{12}\\
-g_{21}&g_{11}
\end{pmatrix}.
$$
Applying the preceding paragraph to the element $g^{-1},$ we conclude that $|g_{22}|=O(|g_{21}|).$

By Theorem II.B.5 in \cite{maskit}, the sequence $\{|g_{21}|\}_{1\neq g\in G}$ is bounded from below. Thus,
$$|g_{12}|=|\frac{g_{11}g_{22}-1}{g_{21}}|\leq\frac{|g_{11}|\cdot|g_{22}|}{|g_{21}|}+\frac1{|g_{21}|}=O(|g_{21}|)+O(1)=O(|g_{21}|).$$

Combining the estimates in the preceding paragraphs, we conclude that $\|g\|_{\infty}=O(|g_{21}|).$ The assertion follows from Lemma \ref{reduction to sullivan}.
\end{proof}

The following lemma provides the converse to Lemma \ref{main kleinian lemma} (under additional assumptions on the group $G$).

\begin{lemma}\label{reduction to sullivan invert} Let $G\subset {\rm PSL}(2,\mathbb{C})$ be as in Theorem \ref{main theorem}. There exists $C>0$ such that
$$\Big\{\frac1{(k+1)^{\frac1p}}\Big\}_{k\geq0}\leq C\mu\Big(\{|g_{21}|^{-2}\}_{1\neq g\in G}\Big).$$
\end{lemma}
\begin{proof} By Theorem 4 of \cite{BE2} the group $G$ is a quasiconformal deformation of a Fuchsian group of the first kind. In particular, its limit set $\Lambda(G)$ is a quasi-circle. By Theorem 12 in \cite{GW}, the Hausdorff dimension of $\Lambda(G)$ is strictly less than $2.$ The group $G$ is finitely generated and thus by the Ahlfors Finiteness Theorem, $G$ is analytically finite. It follows now from Theorem 1.2 in \cite{BJ} that $G$ is geometrically finite. Theorem 1 in \cite{sullivan84} states that the critical exponent equals $p.$ It is proved in \cite{Bowditch} that a geometrically finite Kleinian group without parabolic elements is convex co-compact. Thus, the results of Section 3 in \cite{sullivan} are applicable.

By the left hand side estimate in Corollary 10 in \cite{sullivan}), we have
$${\rm Card}(\{g\in G:\ {\rm dist}((\pi(g))(\mathbf{0}),\mathbf{0})\leq r\})\geq Ce^{pr}.$$
Using the formula \eqref{dist from 0} and denoting $e^{-r}$  by $t,$ we arrive at
$${\rm Card}(\{g\in G:\ \frac{1-|(\pi(g))(\mathbf{0})|}{1+|(\pi(g))(\mathbf{0})|}\geq t\})\geq Ct^{-p}.$$
Since $|(\pi(g))(\mathbf{0})|<1,$ it follows that
$${\rm Card}(\{g\in G:\ 1-|(\pi(g))(\mathbf{0})|^2\geq t\})\geq Ct^{-p}.$$
Since $|a'|=|a|$ and $|c'|=|c|,$ it follows that
$$(\pi(g))(\mathbf{0})\stackrel{\eqref{hyp action}}{=}c'(a')^{-1}\mbox{ and, therefore, }1-|(\pi(g))(\mathbf{0})|^2=1-\frac{|c|^2}{|a|^2}\stackrel{\eqref{sl2 rep h3}}{=}\frac1{|a|^2}.$$
Thus,
$${\rm Card}(\{g\in  G:\ |a|^2\leq t^{-1}\})\geq Ct^{-p}.$$

We infer from \eqref{ac def} that
$$4|a|^2=|g_{11}+\bar{g}_{22}|^2+|g_{12}-\bar{g}_{21}|^2,\quad 4|c|^2=|g_{21}+\bar{g}_{12}|^2+|g_{22}-\bar{g}_{11}|^2.$$
By the parallelogram rule, we have
$$8|a|^2\geq 4|a|^2+4|c|^2=2|g_{11}|^2+2|g_{22}|^2+2|g_{12}|^2+2|g_{21}|^2\geq 2|g_{21}|^2.$$
It follows that
$${\rm Card}(\{g\in  G:\ \frac14|g_{21}|^2\leq t^{-1}\})\geq Ct^{-p}.$$
This concludes the proof.
\end{proof}

\subsection{When does the quantum derivative fall into $\mathcal{L}_{p,\infty}?$}

In this subsection, we find a sufficient condition for the quantum derivative to belong to the ideal $\mathcal{L}_{p,\infty},$ $p>1.$ A similar result for the ideal $\mathcal{L}_p$ is available as Theorem 4 and Proposition 5 on p.~316 in \cite{Connes}. We get the required estimate by real interpolation.

Let $\alpha\neq -1$ and let $\nu_{\alpha}$ be the measure on $\mathbb{D}$ defined by the formula
$$d\nu_{\alpha}(z)=|\alpha+1|(1-|z|^2)^{\alpha}dm(z),$$
where $m$ is the normalised Lebesgue measure on $\mathbb{D}.$ For $\alpha>-1,$ this is a finite measure space; for $\alpha<-1,$ this is infinite measure space. Let ${\rm Hol}(\mathbb{D})$ be the space of all holomorphic functions on $\mathbb{D}.$ The symbol $[\cdot,\cdot]_{\theta,\infty}$ denotes the functor of real interpolation (see e.g. Definition 2.g.12 in \cite{LT2}).

\begin{lemma}\label{first interpolation lemma} If $1<p_0<2,$ then
$$[L_{p_0}(\mathbb{D},\nu_{p_0-2})\cap {\rm Hol}(\mathbb{D}),L_2(\mathbb{D},\nu_0)\cap {\rm Hol}(\mathbb{D})]_{\theta,\infty}=$$
$$=[L_{p_0}(\mathbb{D},\nu_{p_0-2}),L_2(\mathbb{D},\nu_0)]_{\theta,\infty}\cap {\rm Hol}(\mathbb{D}).$$
\end{lemma}
\begin{proof} Clearly, $L_2(\mathbb{D},\nu_0)\cap {\rm Hol}(\mathbb{D})$ is a closed subset in $L_2(\mathbb{D},\nu_0).$ By Proposition 1.2 in \cite{HKZ}, $L_{p_0}(\mathbb{D},\nu_{p_0-2})\cap {\rm Hol}(\mathbb{D})$ is a closed subspace in $L_{p_0}(\mathbb{D},\nu_{p_0-2}),$ so that left hand side is well defined.

The following map (see Proposition 1.4 in \cite{HKZ}) is called Bergman projection.
$$(P_0f)(z)=\int_{\mathbb{D}}\frac{f(w)d\nu_0(w)}{(1-z\bar{w})^2},\quad z\in\mathbb{C}.$$
By Theorem 1.10 in \cite{HKZ}, we have that
$$P_0:L_{p_0}(\mathbb{D},\nu_{p_0-2})\to L_{p_0}(\mathbb{D},\nu_{p_0-2})\cap {\rm Hol}(\mathbb{D})$$
is a bounded mapping. Also, by Theorem 1.10 in \cite{HKZ}, we have that
$$P_0:L_2(\mathbb{D},\nu_0)\to L_2(\mathbb{D},\nu_0)\cap {\rm Hol}(\mathbb{D})$$
is a bounded mapping.

Therefore, for the left hand side of the equality in the statement of Lemma \ref{first interpolation lemma}, we have
$$LHS=[P_0(L_{p_0}(\mathbb{D},\nu_{p_0-2})),P_0(L_2(\mathbb{D},\nu_0))]_{\theta,\infty}=$$
$$=P_0([L_{p_0}(\mathbb{D},\nu_{p_0-2}),L_2(\mathbb{D},\nu_0)]_{\theta,\infty})=[L_{p_0}(\mathbb{D},\nu_{p_0-2}),L_2(\mathbb{D},\nu_0)]_{\theta,\infty}\cap {\rm Hol}(\mathbb{D}).$$
\end{proof}

The following lemma describes the class of functions $f$ on the unit circle $\partial\mathbb{D}$ for which its quantum derivative belongs to the weak ideal $\mathcal{L}_{p,\infty},$ $p>1.$ Here, the function space $L_{p,\infty}(\mathbb{D},\nu_{-2})$ is defined in Subsection \ref{pinfinity}.

\begin{lemma}\label{second interpolation lemma} Suppose $f:\partial\mathbb{D}\to\mathbb{C}$ has an extension to an analytic function on $\mathbb{D}.$ For $p>1,$ we have
$$\|[F,f]\|_{p,\infty}\leq c_p\|h\|_{L_{p,\infty}(\mathbb{D},\nu_{-2})},$$
where $h(z)=(1-|z|^2)|f'(z)|,$ $z\in\mathbb{D}.$
\end{lemma}
\begin{proof} Let $C_p$ be the collection of all $f:\mathbb{D}\to\mathbb{C}$ such that the mapping $z\to(1-|z|^2)f(z),$ $z\in\mathbb{D},$ belongs to the space $L_{p,\infty}(\mathbb{D},\nu_{-2}).$ If $\frac1p=\frac{1-\theta}{p_0}+\frac{\theta}{2},$ then
$$[L_{p_0}(\mathbb{D},\nu_{p_0-2}),L_2(\mathbb{D},\nu_0)]_{\theta,\infty}=C_p.$$

Let
$$A_p^{\frac1p}=\{f\in{\rm Hol}(\mathbb{D}):\ f'\in L_p(\mathbb{D},\nu_{p-2})\}$$
and let
$$D_p=\{f\in{\rm Hol}(\mathbb{D}):\ f'\in C_p\}.$$
It follows from Lemma \ref{first interpolation lemma} that
$$[A_{p_0}^{\frac1{p_0}},A_2^{\frac12}]_{\theta,\infty}=D_p.$$

By Theorem 4 and Proposition 5 on p.~316 in \cite{Connes}, we have
$$[F,f]\in\mathcal{L}_p\Longleftarrow f\in A_p^{\frac1p},\quad 1<p<\infty.$$
Applying real interpolation method to the Banach couples $(A_{p_0}^{\frac1{p_0}},A_2^{\frac12})$ and $(\mathcal{L}_{p_0},\mathcal{L}_2),$ we infer
$$\|[F,f]\|_{p,\infty}=\|[F,f]\|_{[\mathcal{L}_{p_0},\mathcal{L}_2]_{\theta,\infty}}\leq c_p\|f\|_{[A_{p_0}^{\frac1{p_0}},A_2^{\frac12}]_{\theta,\infty}}=\|f\|_{D_p}.$$
\end{proof}

\subsection{Proof of Theorem \ref{main theorem} \eqref{mainta}}

We are now ready to prove the first part of our main result.

\begin{proof}[Proof of Theorem \ref{main theorem} \eqref{mainta}] As explained in the (first few lines of the) proof of Lemma \ref{reduction to sullivan invert}, the group $G$ is geometrically finite. By Theorem 1 in \cite{sullivan84}, the critical exponent $\delta$ equals to the Hausdorff dimension $p$ of $\Lambda(G).$ Note that $p>1$ by Theorem 2 in \cite{BO}.

Consider $G$ acting on $\Sigma_{{\rm int}}.$ Let $\pi$ be the action of $G$ on the unit disk by the formula
\begin{equation}\label{fuchs action}
g\circ Z=Z\circ\pi(g).
\end{equation}
Since every $\pi(g)$ is a conformal automorphism of the unit disk, it is automatically fractional linear (see \cite{markushevich}). Thus, $\pi(G)$ is a group of fractional linear transformations preserving the unit circle, i.e. a Fuchsian group and it's limit set is the unit circle $\partial\mathbb{D}$, thus it is Fuchsian of the first kind.  As a group, $\pi(G)$ is isomorphic to $G$ and is, therefore, finitely generated.



We claim that the Fuchsian group $\pi(G)$ does not contain  parabolic elements. Assume the contrary: let $g\in G$ be such that $\pi(g)$ is parabolic. Hence, there exists a fixed point $w_0\in\partial\mathbb{D}$ of $\pi(g)$ such that $(\pi(g))^nw\to w_0$ as $n\to\pm\infty$ for every $w\in\mathbb{D}.$ Let $w=Z(z),$ $z\in\Sigma_{{\rm int}},$ and let $w_0=Z(z_0),$ $z_0\in\Lambda(G).$ By \eqref{fuchs action}, we have that $g^n(z)\to z_0$ as $n\to\pm\infty.$ Hence, $g\in G$ is parabolic,\footnote{ An element $g\in{\rm PSL}(2,\mathbb{C})$ is either parabolic or diagonalizable. If $g$ is diagonalizable, then (after conjugating $g$ by a fractional linear transform), we have that $g:z\to az$ for every $z\in\mathbb{C}.$ If $|a|<1,$ then $g^n(z)\to0$ as $n\to\infty$ and $g^n(z)\to\infty$ as $n\to-\infty$ for every $0\neq z\in\mathbb{C}.$ If $|a|>1,$ then $g^n(z)\to0$ as $n\to-\infty$ and $g^n(z)\to\infty$ as $n\to\infty$ for every $0\neq z\in\mathbb{C}.$ If $|a|=1$ and $a\neq 1,$ then the sequence $\{g^n(z)\}_{n\in\mathbb{Z}}$ diverges as $n\to\infty$ and as $n\to-\infty$ for every $0\neq z\in\mathbb{C}.$} which is not the case.

Since $\pi(G)$ is finitely generated and of the first kind, it follows from Theorem 10.4.3 in \cite{Beardon} that the Riemann surface $\mathbb{D}/\pi(G)$ has finite area. Taking into account that $\pi(G)$ does not have parabolic elements, we infer from Corollary 4.2.7 in \cite{Katok} that the Riemann surface $\mathbb{D}/\pi(G)$ is compact. By Corollary 4.2.3 and Theorem 3.2.2 in \cite{Katok}, $\pi(G)$ admits a fundamental domain $\mathbb{F}$ which is compactly supported in $\mathbb{D}.$

{\bf Step 1:} We claim that there exists a finite constant  such that for every $g\in G,$
$$\sup_{z\in \pi(g)\mathbb{F}}(1-|z|^2)|Z'(z)|\leq\frac{{\rm const}}{|g_{21}|^2}.$$

Indeed, we have $z=\pi(g)w,$ where $w\in\mathbb{F}.$ We have\footnote{This is a standard fact. Let $k:w\to\frac{\alpha w+\beta}{\bar{\beta}w+\bar{\alpha}},$ $|\alpha|^2-|\beta|^2=1$ be an arbitrary conformal automorphism of the unit disk. We have
$$\frac{|dk(w)|}{1-|k(w)|^2}=\frac{|\bar{\beta}w+\bar{\alpha}|^{-2}}{1-\frac{|\alpha w+\beta|^2}{|\bar{\beta}w+\bar{\alpha}|^2}}|dw|=\frac{|dw|}{|\bar{\beta}w+\bar{\alpha}|^2-|\alpha w+\beta|^2}=\frac{|dw|}{1-|w|^2}.$$}
$$1-|z|^2=(1-|w|^2)|(\pi(g))'(w)|.$$
It follows from the chain rule that
$$(1-|z|^2)|Z'(z)|=(1-|w|^2)\cdot |Z'(\pi(g)w)|\cdot |(\pi(g))'(w)|=(1-|w|^2)\cdot|(Z\circ\pi(g))'(w)|.$$
It follows from \eqref{fuchs action} and chain rule that
\begin{equation}\label{127}
(1-|z|^2)|Z'(z)|\stackrel{\eqref{fuchs action}}{=}(1-|w|^2)\cdot|(g\circ Z)'(w)|=|g'(Z(w))|\cdot (1-|w|^2)|Z'(w)|.
\end{equation}
Since $g'(u)=(g_{21}u+g_{22})^{-2}$ and $g^{-1}(\infty)=-\frac{g_{22}}{g_{21}},$ it follows that
\begin{equation}\label{129}
|g'(Z(w))|=\frac1{|g_{21}Z+g_{22}|^2}=\frac1{|g_{21}|^2}\cdot \frac1{|Z(w)-g^{-1}(\infty)|^2}.
\end{equation}
Thus, for $z\in\pi(g)\mathbb{F},$ we have, since $g^{-1}(\infty)$ stays in the unbounded component of the complement of the limit set $\Lambda(G)$ and thus $|Z(w)-g^{-1}(\infty)|\geq {\rm dist}(Z(\mathbb{F}),\Lambda(G)),$
$$(1-|z|^2)|Z'(z)|\leq \frac1{|g_{21}|^2}\cdot\frac1{{\rm dist}^2(Z(\mathbb{F}),\Lambda(G))}\cdot\sup_{w\in\mathbb{F}}(1-|w|^2)|Z'(w)|.$$
Since $\mathbb{F}$ is compact and $Z'|_{\mathbb{F}}$ is continuous, the claim follows.

{\bf Step 2:} Let $h(z)=(1-|z|^2)|Z'(z)|$ (see also the statement of Lemma \ref{second interpolation lemma}). It follows from Step 1 that
$$\|h\|_{L_{p,\infty}(\mathbb{D},\nu_{-2})}\leq\|h\chi_{\mathbb{F}}\|_{L_{p,\infty}(\mathbb{D},\nu_{-2})}+{\rm const}\cdot\|\sum_{1\neq g\in G}\frac1{|g_{21}|^2}\chi_{\pi(g)\mathbb{F}}\|_{L_{p,\infty}(\mathbb{D},\nu_{-2})}.$$
Recall that $\mathbb{F}$ is compactly supported in $\mathbb{D}$ and, therefore, $\nu_{-2}(\mathbb{F})<\infty.$ Let $\nu_{-2}(\mathbb{F})=a.$ Elements of the group $\pi(G)$ are conformal automorphisms of the unit disk; hence, isometries of the hyperbolic plane $\mathbb{H}^2.$ The measure $\nu_{-2}$ is a volume form of $\mathbb{H}^2$ and is, therefore, invariant with respect to its isometries. Hence, $\nu_{-2}$ is $\pi(G)-$invariant.\footnote{This fact can also be seen directly as follows. Let $k:z\to\frac{\alpha z+\beta}{\bar{\beta}z+\bar{\alpha}},$ $|\alpha|^2-|\beta|^2=1$ be an arbitrary conformal automorphism of the unit disk. Its Jacobian is exactly $|k'(z)|^2.$ Thus,
$$d(\nu_{-2}\circ k)(z)=\frac{d(m\circ k)(z)}{(1-|k(z)|^2)^2}=\frac{|k'(z)|^2}{(1-|k(z)|^2)^2}dm(z)=\frac{dm(z)}{(1-|z|^2)^2}=d\nu_{-2}(z).$$
This shows conformal invariance of the measure $\nu_{-2}.$} It follows that
\begin{equation}\label{131}
\nu_{-2}(\pi(g)\mathbb{F})=a,\quad\mbox{for every }g\in G.
\end{equation}
Thus,
$$\mu(\sum_{1\neq g\in G}\frac1{|g_{21}|^2}\chi_{\pi(g)\mathbb{F}})=\mu\Big(\Big\{\frac1{|g_{21}|^2}\Big\}_{1\neq g\in G}\otimes\chi_{(0,a)}\Big),$$
where $\mu$ on the left hand side is computed in the measure space $(\mathbb{D},\nu_{-2})$ and $\mu$ on the right hand side is computed in the algebra $(G\times (0,\infty),{\rm Card}\times m).$ Hence,
$$\|h\|_{L_{p,\infty}(\mathbb{D},\nu_{-2})}\leq\|h\chi_{\mathbb{F}}\|_{\infty}\|\chi_{(0,a)}\|_{p,\infty}+{\rm const}\cdot\Big\|\Big\{\frac1{|g_{21}|^2}\Big\}_{1\neq g\in G}\otimes\chi_{(0,a)}\Big\|_{p,\infty}.$$
It follows now from Lemma \ref{main kleinian lemma} that $h\in L_{p,\infty}(\mathbb{D},\nu_{-2}).$ The assertion follows now from Lemma \ref{second interpolation lemma}.
\end{proof}

The next lemma is the core part of the proof of Theorem \ref{main theorem} \eqref{maintc}. Its proof is similar to that of Theorem \ref{main theorem} \eqref{mainta}.

\begin{lemma}\label{maintc core} If $G$ is as in Theorem \ref{main theorem}, then
$$\liminf_{s\to0}s\|[F,Z]\|_{p+s}>0.$$
\end{lemma}
\begin{proof} Let $h(z)=(1-|z|^2)|Z'(z)|,$ $z\in\mathbb{D}.$ For every $1\neq g\in G,$ it follows from \eqref{127} and \eqref{129} (in the proof of Theorem \ref{main theorem} \eqref{mainta}) that
$$\inf_{z\in \pi(g)\mathbb{F}}(1-|z|^2)|Z'(z)|\geq\frac1{|g_{21}|^2}\cdot\inf_{w\in\mathbb{F}}(1-|w|^2)|Z'(w)|\cdot\inf_{w\in\mathbb{F}}\frac1{|Z(w)-g^{-1}(\infty)|^2}\geq$$
$$\geq \frac1{|g_{21}|^2}\cdot\inf_{w\in\mathbb{F}}(1-|w|^2)|Z'(w)|\cdot\frac1{(\|Z\|_{\infty}+|g^{-1}(\infty)|)^2}\geq \frac{{\rm const}}{|g_{21}|^2}.$$
We have $\nu_{-2}(\pi(g)\mathbb{F})=a$ for every $g\in G.$ Since the sets $\{\pi(g)\mathbb{F}\}_{g\in G}$ are pairwise disjoint, it follows that
$$\|h\|_{L_{p+s}(\mathbb{D},\nu_{-2})}\geq{\rm const}\cdot\|\{|g_{21}|^{-2}\}_{1\neq g\in G}\|_{p+s}.$$
We infer from Lemma \ref{reduction to sullivan invert} that
$$\|\{|g_{21}|^{-2}\}_{1\neq g\in G}\|_{p+s}\geq{\rm const}\cdot\|\{(k+1)^{-\frac1p}\}_{k\geq0}\|_{p+s}\geq\frac{{\rm const}}{s},\quad s\downarrow0.$$

By Proposition 5 on p. 316 in \cite{Connes}, we have
$$\|Z\|_{A_{p+s}^{\frac1{p+s}}}\geq c_p\|h\|_{L_{p+s}(\mathbb{D},\nu_{-2})}\geq\frac{{\rm const}}{s},\quad s\downarrow0.$$
Since $Z$ is an analytic function on $\mathbb{D},$ it follows from Theorem 4 on p.~316 in \cite{Connes} that
$$\|[F,Z]\|_{p+s}\geq c_p^{-1}\|Z\|_{B_{p+s}^{\frac1{p+s}}}=c_p^{-1}\|Z\|_{A_{p+s}^{\frac1{p+s}}}\geq\frac{{\rm const}}{s},\quad s\downarrow0.$$
This completes the proof.
\end{proof}

\section{Integration in $(\mathcal{L}_{p,\infty})_0$, $p>1.$}\label{measurability section}

\begin{lemma}\label{first measurability lemma} Let $s\to Z(s)$ be a bounded function from $\mathbb{R}$ to $(\mathcal{L}_{p,\infty})_0.$ If it is measurable in the weak operator topology, then it is weakly measurable\footnote{See Definition \ref{bweak}} in $(\mathcal{L}_{p,\infty})_0.$
\end{lemma}
\begin{proof} Let $\gamma$ be a bounded linear functional on $(\mathcal{L}_{p,\infty})_0.$ By the noncommutative Yosida-Hewitt theorem (see \cite{Peter}), we have that $\gamma$ extends to a normal functional on $\mathcal{L}_{p,\infty}.$ Let $\mathcal{L}_{q,1}$ be the Lorentz space which is the K\"othe dual\footnote{See \cite{Peter} for the definition and basic properties of K\"othe duals.} of $\mathcal{L}_{p,\infty}.$ There exists $x\in\mathcal{L}_{q,1}$ such that
$$\gamma(y)={\rm Tr}(xy),\quad y\in\mathcal{L}_{p,\infty}.$$
Fix $n\in\mathbb{N}$ and choose a finite rank operator $x_n$ such that $\|x-x_n\|_{q,1}<\frac1n.$ By assumption, the scalar valued function
$$f_n:s\to {\rm Tr}(x_nZ(s)),\quad s\in\mathbb{R},$$
is measurable. On the other hand, we have
$$|f-f_n|(s)\leq\|Z(s)\|_{p,\infty}\|x-x_n\|_{q,1}$$
and, therefore,
$$\|f-f_n\|_{\infty}\leq\frac1n\sup_{s\in\mathbb{R}}\|Z(s)\|_{p,\infty}.$$
Hence, $f_n$ converges to $f$ uniformly. Since the limit of a sequence of measurable functions is measurable, the weak measurability of the mapping $s\to Z(s)$ follows.
\end{proof}

\begin{lemma}\label{second measurability lemma} Let $s\to Z(s)$ be a bounded function from $\mathbb{R}$ to $(\mathcal{L}_{p,\infty})_0$ which is measurable in the weak operator topology. If
\begin{equation}\label{integrability lpinfty}
\int_{\mathbb{R}}\|Z(s)\|_{p,\infty}ds<\infty,
\end{equation}
then $s\to Z(s)$ is a Bochner integrable function from $\mathbb{R}$ to $(\mathcal{L}_{p,\infty})_0.$ We have
\begin{equation}\label{integral lpinfty}
\int_{\mathbb{R}}Z(s)ds\in(\mathcal{L}_{p,\infty})_0.
\end{equation}
\end{lemma}
\begin{proof} By Lemma \ref{first measurability lemma} the mapping $s\to Z(s)$ is weakly measurable from $\mathbb{R}$ to $(\mathcal{L}_{p,\infty})_0.$ Since $(\mathcal{L}_{p,\infty})_0$ is separable, it follows from Theorem 3.5.3 in \cite{hf} that the mapping $s\to Z(s)$ is strongly measurable from $\mathbb{R}$ to $(\mathcal{L}_{p,\infty})_0$ (in the sense of Definition 3.5.4 in \cite{hf}). Using Theorem 3.7.4 in \cite{hf} and \eqref{integrability lpinfty}, we obtain that the mapping $s\to Z(s)$ is Bochner integrable from $\mathbb{R}$ to $(\mathcal{L}_{p,\infty})_0.$ The inclusion \eqref{integral lpinfty} follows now from the definition of Bochner integral (see Definition 3.7.3 in \cite{hf}).
\end{proof}

In what follows, we use the notation $A^{z}$ for the complex power of a positive operator $A\in\mathcal{L}(H)$ defined as follows for $z\in \mathbb{C}$ of positive real part: $\Re(z)\geq 0$. Let $f_z:[0,\infty)\to\mathbb{C}$ be the Borel function given by the formula
$$
f_z(x)=
\begin{cases}
e^{z\log(x)},\quad x>0\\
0,\quad x=0.
\end{cases}
$$
We set $A^{z}=f_z(A),$ where the right hand side is defined by means of the functional calculus. In particular this defines the imaginary power $A^{is}=f_{is}(A)$ for $s\in\mathbb{R}$. One has $A^{z+z'}=A^{z}A^{z'}$ for  $z,z'\in \mathbb{C}$ of positive real part $\Re(z)\geq 0$, $\Re(z')\geq 0$.

One has $f_z(xy)=f_z(x)f_z(y)$ for $\Re(z)\geq 0$ and $x,y\geq 0$. Thus using the convention $0^{z}=0$ for $z\in \mathbb{C}$,  $\Re(z)\geq 0$, (in particular $0^{is}=0$) one has the formula
$$\lambda^{is}A^{is}=(\lambda A)^{is},\quad s\in\mathbb{R},\quad \lambda\geq0,\quad 0\leq A\in\mathcal{L}(H),$$
which is used repeatedly in Lemmas \ref{first integral formula} and \ref{second integral formula}.

\begin{lemma}\label{third measurability lemma} Let $A_1,A_2,A_3\in\mathcal{L}(H)$ be positive and let $X_1,X_2,X_3,X_4\in\mathcal{L}(H).$ The mapping
$$s\to X_1A_1^{is}X_2A_2^{is}X_3A_3^{is}X_4,\quad s\in\mathbb{R},$$
is measurable in the weak operator topology.
\end{lemma}
\begin{proof} For every bounded positive operator $A,$ the mapping $s\to A^{is}$ is strongly continuous. Indeed, let $\log_{{\rm fin}}$ be a Borel function on $[0,\infty)$ defined by the formula
$$\log_{{\rm fin}}(x)=
\begin{cases}
\log(x),\quad x>0\\
0,\quad x=0
\end{cases}.
$$
We have that $\log_{{\rm fin}}(A)$ is an unbounded self-adjoint operator. Thus, the mapping
$$s\to A^{is}\stackrel{}{=}\exp(is\log_{\rm fin}(A))\cdot E_A(0,\infty)$$
is strongly continuous by Stone's theorem.

Thus, for arbitrary vectors $\xi,\eta\in H,$ the mapping
$$s\to\langle X_1A_1^{is}X_2A_2^{is}X_3A_3^{is}X_4\xi,\eta\rangle,\quad s\in\mathbb{R},$$
is continuous. In particular, the latter scalar-valued mapping is measurable and our vector-valued mapping is measurable in the weak operator topology.
\end{proof}

\section{Proof of the key \lq\lq commutator\rq\rq estimate}

This section contains a modification of Lemma 11 stated on p. 321 in \cite{Connes}. The proofs here were obtained with the help of Denis Potapov.

In this section, integrals are understood in the weak sense (see Subsection \ref{weak int}) unless explicitly specified otherwise.

\begin{lemma}\label{first integral formula} For every $p>1,$ there exists a Schwartz function $h$ such that, for every $0\leq X,Y\in\mathcal{L}(H)$ we have
$$X^p-Y^p=V-\int_{\mathbb{R}}X^{is}VY^{-is}h(s)ds.$$
Here, $V=X^{p-1}(X-Y)+(X-Y)Y^{p-1}.$
\end{lemma}
\begin{proof} Define a  function $g$ by setting
$$g(t)=1-\frac{e^{\frac{p}{2}t}-e^{-\frac{p}{2}t}}{(e^\frac{t}{2}-e^{-\frac{t}{2}})(e^{(\frac{p-1}{2})t}+e^{-(\frac{p-1}{2})t})},\quad t\in\mathbb{R}, t\neq 0, \ \ g(0)=\left(1-\frac{p}{2}\right).$$
It is an even function of $t$, it is smooth at $t=0$ with Taylor expansion 
$$g(t)=\left(1-\frac{p}{2}\right)+\frac{1}{24} \left(p^3-3 p^2+2 p\right) t^2+\cdots$$
and one has 
$$
g(t)=\frac{e^{2 t}-e^{p t}}{\left(e^t-1\right) \left(e^{p t}+e^t\right)}
$$
so that $g=0$ for $p=2$, and  $g(t)$ is equivalent to $e^{(1-p)t}$ when $t\to \infty$ for $p<2$, and to $-e^{-t}$ for $p>2$. Similarly all derivatives of $g$ have exponential decay at $\infty$. Thus $g$ is a Schwartz function. 
Set $h$ to be the Fourier transform of $g,$ so that $h$ is also a Schwartz function. Set
$$\phi_1(\lambda,\mu)=g(\log(\frac{\lambda}{\mu})) \quad\forall  \lambda,\mu>0,
\ \phi_1(0,\mu)=0, \ \forall \mu\geq 0, \ \phi_1(\lambda,0)=0, \forall \lambda\geq 0.$$
So that our function $\phi_1$ is defined on $[0,\infty)\times[0,\infty).$ Note that it is not continuous at $(0,0)$.
One has 
\begin{equation}\label{phi1 alg}\phi_1(\lambda,\mu)=1-\frac{\lambda^p-\mu^p}{(\lambda-\mu)(\lambda^{p-1}+\mu^{p-1})},\quad \lambda,\mu>0, \lambda\neq \mu.\end{equation}
We claim that
\begin{equation}\label{phi1 rep}
\phi_1(\lambda,\mu)=\int_{\mathbb{R}}h(s)\lambda^{is}\mu^{-is}ds,\quad\lambda,\mu\geq0.
\end{equation}
Indeed, we have
$$g(t)=\int_{\mathbb{R}}h(s)e^{ist}ds,\quad t\in\mathbb{R}.$$
For $\lambda,\mu>0,$ we set $t=\log(\frac{\lambda}{\mu})$ and obtain
$$\phi_1(\lambda,\mu)=\int_{\mathbb{R}}h(s)e^{is\log(\frac{\lambda}{\mu})}ds=\int_{\mathbb{R}}h(s)\lambda^{is}\mu^{-is}ds$$
For $\lambda=0$ or $\mu=0,$ the left hand side of \eqref{phi1 rep} vanishes by the definition of $\phi_1,$ while the right hand side vanishes due to the convention $0^{is}=0.$ Thus, formula \eqref{phi1 rep} holds for all $\lambda,\mu\geq0.$ 
Set
$$\phi_2(\lambda,\mu)=(\lambda^{p-1}+\mu^{p-1})(\lambda-\mu),\quad\lambda,\mu\geq0.$$
This function is bounded on ${\rm Spec}(X)\times {\rm Spec}(Y)$ and the same holds for 
$$\phi_3(\lambda,\mu)=(\lambda^{p-1}+\mu^{p-1})(\lambda-\mu)-(\lambda^p-\mu^p),\quad \forall \lambda,\mu \geq 0.$$
The equality $\phi_3=\phi_1\phi_2$ holds on $[0,\infty)\times [0,\infty)$. Indeed this follows from \eqref{phi1 alg} for $\lambda,\mu>0$, $\lambda\neq \mu$. For $\lambda=\mu>0$ one has $\phi_1(\lambda,\lambda)=1-\frac{p}{2}$,  $\phi_2(\lambda,\lambda)=0$ and $\phi_3(\lambda,\lambda)=0$. If $\lambda=0$ or $\mu=0$ one has $\phi_1(\lambda,\mu)=0$ and $\phi_3(\lambda,\mu)=0$.

It follows from the definition \eqref{123} of Double Operator Integrals and $X,Y\geq 0$,
\begin{equation}\label{doi usual}
T_{\phi_1}^{X,Y}(A)=\int_{\mathbb{R}}h(s)X^{is}AY^{-is}ds.
\end{equation}
Indeed, since $h$ is a Schwartz function, the condition \eqref{119} holds and, therefore, \eqref{123} reads as \eqref{doi usual}. Here, the integral on the right hand side is understood in the weak sense. Measurability of the integrand is guaranteed by Lemma \ref{third measurability lemma} and condition \eqref{necessary-condition} follows from the inequality
$$\|h(s)X^{is}AY^{-is}\|_{\infty}\leq|h(s)|\cdot\|A\|_{\infty},\quad s\in\mathbb{R},$$
and from the fact that $h$ is a Schwartz (and, hence, integrable) function. In particular, $T^{X,Y}_{\phi_1}:\mathcal{L}_{\infty}\to\mathcal{L}_{\infty}.$

Using formulae \eqref{integral tensor product} and \eqref{123}, we obtain that $T^{X,Y}_{\phi_2}:\mathcal{L}_{\infty}\to\mathcal{L}_{\infty}$ and
$$T^{X,Y}_{\phi_2}(A)=X^pA-X^{p-1}AY+XAY^{p-1}-AY^p.$$
The function $\phi_3$ bounded on ${\rm Spec}(X)\times {\rm Spec}(Y)$, $T^{X,Y}_{\phi_3}:\mathcal{L}_{\infty}\to\mathcal{L}_{\infty}$ and
$$T^{X,Y}_{\phi_3}(A)=(X^pA-X^{p-1}AY+XAY^{p-1}-AY^p)-(X^pA-AY^p).$$

We have $\phi_3=\phi_1\phi_2$ on ${\rm Spec}(X)\times {\rm Spec}(Y)$, and thus
$$T^{X,Y}_{\phi_1}(V)\stackrel{\eqref{doi algebraic}}{=}T^{X,Y}_{\phi_1}(T^{X,Y}_{\phi_2}(1))=T^{X,Y}_{\phi_3}(1)=V-(X^p-Y^p).$$
The assertion follows now from \eqref{doi usual}.
\end{proof}

Lemma \ref{second integral formula} below can be proved without any compactness assumption on the operator $B$; however, the proof becomes much harder. We impose compactness assumption due to the fact that $B$ is compact in Lemma \ref{main lemma} (the only place where we use Lemma \ref{second integral formula}).

\begin{lemma}\label{second integral formula} Let $0\leq A,B\in\mathcal{L}(H).$ If $1<p<\infty$ and if $B$ is compact, then
$$B^pA^p-(A^{\frac12}BA^{\frac12})^p="T(0)"-\int_{\mathbb{R}}T(s)h(s)ds,$$
where we denote, for brevity, $Y=A^{\frac12}BA^{\frac12}$ while
$$T(s)=B^{p-1+is}[B,A^{p+is}]Y^{-is}+B^{p-1+is}A^{p-\frac12+is}[A^{\frac12},B]Y^{-is}+$$
$$+B^{is}[B,A^{1+is}]Y^{p-1-is}+B^{is}A^{\frac12+is}[A^{\frac12},B]Y^{p-1-is}.$$
and 
$$"T(0)":=B^{p-1}[B,A^p]+B^{p-1}A^{p-\frac12}[A^{\frac12},B]+[B,A]Y^{p-1}+A^{\frac12}[A^{\frac12},B]Y^{p-1}.
 $$
\end{lemma}
\begin{proof} By assumption, $B$ is compact and, therefore, one can write $B=\sum_j\lambda_jp_j,$ where $\{p_j\}$ is a family of mutually orthogonal projections such that $\sum_jp_j=1.$ We have
$$B^pA^p-Y^p=\sum_jp_j(B^pA^p-Y^p)=\sum_jp_j((\lambda_jA)^p-Y^p).$$
Applying Lemma \ref{first integral formula} to the expression in the brackets, we obtain\footnote{In this and subsequent formulae, imaginary powers are defined as in Section \ref{measurability section}. The convention $0^{is}=0$ is used.}
\begin{equation}\label{116}
B^pA^p-Y^p=\sum_jp_j(V_j-\int_{\mathbb{R}}(\lambda_jA)^{is}V_jY^{-is}h(s)ds),
\end{equation}
where,
$$V_j=(\lambda_jA)^{p-1}(\lambda_jA-Y)+(\lambda_jA-Y)Y^{p-1}=(\lambda_jA)^p-(\lambda_jA)^{p-1}Y+\lambda_jAY^{p-1}-Y^p.$$
Therefore, we get $\sum_jp_jV_j=B^pA^p-B^{p-1}A^{p-1}Y+BAY^{p-1}-Y^p="T(0)"$. Moreover we have
$$\sum_jp_j(\lambda_jA)^{is}V_j=\sum_jp_j\Big((\lambda_jA)^{p+is}-(\lambda_jA)^{p-1+is}Y+(\lambda_jA)^{1+is}Y^{p-1}-(\lambda_jA)^{is}Y^p\Big)=$$
$$=\sum_jp_j\lambda_j^{p+is}\cdot A^{p+is}-\sum_jp_j\lambda_j^{p-1+is}\cdot A^{p-1+is}Y+$$
$$+\sum_jp_j\lambda_j^{1+is}\cdot A^{1+is}Y^{p-1}-\sum_jp_j\lambda_j^{is}\cdot A^{is}Y^p.$$
By the functional calculus, we have
$$\sum_jp_j(\lambda_jA)^{is}V_j=B^{p+is}A^{p+is}-B^{p-1+is}A^{p-1+is}Y+B^{1+is}A^{1+is}Y^{p-1}-B^{is}A^{is}Y^p=$$
$$=B^{p-1+is}(BA^{p+is}-A^{p-1+is}Y)+B^{is}(BA^{1+is}-A^{is}Y)Y^{p-1}=$$
$$=B^{p-1+is}[B,A^{p+is}]+B^{p-1+is}A^{p-1+is}(AB-Y)+$$
$$+B^{is}[B,A^{1+is}]Y^{p-1}+B^{is}A^{is}(AB-Y)Y^{p-1}=$$
$$=B^{p-1+is}[B,A^{p+is}]+B^{p-1+is}A^{p-1+is}A^{\frac12}[A^{\frac12},B]+$$
$$+B^{is}[B,A^{1+is}]Y^{p-1}+B^{is}A^{is}A^{\frac12}[A^{\frac12},B]Y^{p-1}.$$
Substituting the last equality into \eqref{116} completes the proof.
\end{proof}

The following lemma is the main result of this section. It provides the key estimate used in the proof of Theorem \ref{main theorem} \eqref{maintb}. In \cite{Connes}, the corresponding Lemma 3$.\beta.$11 is stated without a proof.

\begin{lemma}\label{main lemma} Let $0\leq A\in\mathcal{L}_{\infty}$ and let $0\leq B\in\mathcal{L}_{p,\infty},$ $1<p<\infty.$ If  $[A^{\frac12},B]\in(\mathcal{L}_{p,\infty})_0,$ then
$$B^pA^p-(A^{\frac12}BA^{\frac12})^p\in(\mathcal{L}_{1,\infty})_0.$$
\end{lemma}
\begin{proof}  Consider the formula for $B^pA^p-(A^{\frac12}BA^{\frac12})^p$ obtained in Lemma \ref{second integral formula}. We have
$$B^pA^p-(A^{\frac12}BA^{\frac12})^p="T(0)"-B^{p-1}\cdot(I+II)-(III+IV)\cdot Y^{p-1},$$
where
$$I=\int_{\mathbb{R}}B^{is}[B,A^{p+is}]Y^{-is}h(s)ds,\quad II=\int_{\mathbb{R}}B^{is}A^{p-\frac12+is}[A^{\frac12},B]Y^{-is}h(s)ds,$$
$$III=\int_{\mathbb{R}}B^{is}[B,A^{1+is}]Y^{-is}h(s)ds,\quad IV=\int_{\mathbb{R}}B^{is}A^{\frac12+is}[A^{\frac12},B]Y^{-is}h(s)ds.$$

{\bf Step 1:} We show that $I\in(\mathcal{L}_{p,\infty})_0.$

Without loss of generality, $0\leq A\leq 1.$ For a fixed $s\in\mathbb{R},$ the function $x\to x^{p+is}$ can be uniformly approximated by polynomials $f_m$ on the interval $[0,1].$ It is immediate that
$$[B,A^{p+is}]-[B,f_m(A)]=B(A^{p+is}-f_m(A))-(A^{p+is}-f_m(A))B.$$
Thus,
$$\|[B,A^{p+is}]-[B,f_m(A)]\|_{p,\infty}\leq 2\|B\|_{p,\infty}\|A^{p+is}-f_m(A)\|_{\infty}\to0,\quad m\to\infty.$$
Due to the assumption $[A^{\frac12},B]\in(\mathcal{L}_{p,\infty})_0,$ we have
$$[B,A]=A^{\frac12}[B,A^{\frac12}]+[B,A^{\frac12}]A^{\frac12}\in(\mathcal{L}_{p,\infty})_0.$$
Thus,
$$[B,A^k]=\sum_{l=0}^{k-1}A^l[B,A]A^{k-1-l}\in(\mathcal{L}_{p,\infty})_0.$$
It follows that $[B,f_m(A)]\in(\mathcal{L}_{p,\infty})_0.$ Thus,
\begin{equation}\label{commut conclusion}
[B,A^{p+is}]\in(\mathcal{L}_{p,\infty})_0.
\end{equation}

By hypothesis, one has $B\in\mathcal{L}_{p,\infty}.$ We infer from $0\leq A\leq 1$ that $A^{p+is}$ is a contraction for every $s\in\mathbb{R}.$ Hence, we have
\begin{equation}\label{117}
\|[B,A^{p+is}]\|_{p,\infty}\leq 2\|B\|_{p,\infty}\|A^{p+is}\|_{\infty}\leq 2\|B\|_{p,\infty}.
\end{equation}

It follows from Lemma \ref{third measurability lemma} that the mapping
$$s\to B^{is}[B,A^{p+is}]Y^{-is}h(s),\quad s\in\mathbb{R},$$
is measurable in the weak operator topology. 
Combining Lemma \ref{second measurability lemma} and \eqref{117}, we infer that $I\in(\mathcal{L}_{p,\infty})_0.$

{\bf Step 2:} By Step 1, we have that $I\in(\mathcal{L}_{p,\infty})_0.$ Repeating the argument in Step 1 for $III$ and using $[A^{\frac12},B]\in(\mathcal{L}_{p,\infty})_0,$ for $II$ and $IV$, we obtain that also $II,III,IV\in(\mathcal{L}_{p,\infty})_0.$

The next assertion is similar to \eqref{lpi mult} and it follows immediately from Corollary 2.3.16.b in \cite{LSZ}: if $X\in(\mathcal{L}_{p,\infty})_0$ and $0\leq Z\in\mathcal{L}_{p,\infty},$ then $XZ^{p-1}\in(\mathcal{L}_{1,\infty})_0$ and $Z^{p-1}X\in(\mathcal{L}_{1,\infty})_0.$ Since $B,Y\in\mathcal{L}_{p,\infty},$ it follows that
$$B^{p-1}\cdot(I+II)\in(\mathcal{L}_{1,\infty})_0,\quad (III+IV)\cdot Y^{p-1}\in(\mathcal{L}_{1,\infty})_0.$$

Also, we have by Lemma \ref{second integral formula}
$$"T(0)"=B^{p-1}[B,A^p]+B^{p-1}A^{p-\frac12}[A^{\frac12},B]+[B,A]Y^{p-1}+A^{\frac12}[A^{\frac12},B]Y^{p-1}.$$
Setting $s=0$ in \eqref{commut conclusion}, we obtain that $[B,A^p]\in(\mathcal{L}_{p,\infty})_0.$ By the commutator assumption and Leibniz rule, we have
$$[B,A]=[B,A^{\frac12}]A^{\frac12}+A^{\frac12}[B,A^{\frac12}]\in(\mathcal{L}_{p,\infty})_0.$$
Since $B,Y\in\mathcal{L}_{p,\infty},$ it follows that $"T(0)"\in(\mathcal{L}_{1,\infty})_0.$

Combining these results, we complete the proof.
\end{proof}

\section{Proof of Theorem \ref{main theorem} \eqref{maintb}}

For a detailed study of commutator estimates for the absolute value function, we refer the reader to \cite{DDPS} or \cite{CPSZ}.

\begin{lemma}\label{absolute value lemma} Let $A,B\in\mathcal{L}(H).$ If $[A,B]\in(\mathcal{L}_{p,\infty})_0$ and $[A,B^*]\in(\mathcal{L}_{p,\infty})_0$ then $[A,|B|]\in(\mathcal{L}_{p,\infty})_0.$
\end{lemma}
\begin{proof} For a self-adjoint $B,$ the assertion is proved in \cite{DDPS}. Let $B\in\mathcal{L}(H)$ be arbitrary and set
$$C=
\begin{pmatrix}
A&0\\
0&A
\end{pmatrix},\quad
D=
\begin{pmatrix}
0&B\\
B^*&0
\end{pmatrix}
$$
We have
$$[C,D]=
\begin{pmatrix}
0&[A,B]\\
[A,B^*]&0
\end{pmatrix}
\in(\mathcal{L}_{p,\infty})_0.$$
Since $D$ is self-adjoint, it follows from Theorem 3.4 in \cite{DDPS} that $[C,|D|]\in(\mathcal{L}_{p,\infty})_0.$ However,
$$|D|=
\begin{pmatrix}
|B^*|&0\\
0&|B|
\end{pmatrix}
$$
Thus,
$$[C,|D|]=
\begin{pmatrix}
[A,|B^*|]&0\\
0&[A,|B|]
\end{pmatrix}
$$
This concludes the proof.
\end{proof}

The following lemma is Proposition 10, part (3) on p. 320 in \cite{Connes}.

\begin{lemma}\label{bks lemma} If $T,S\in\mathcal{L}_{p,\infty}$ are such that $T-S\in(\mathcal{L}_{p,\infty})_0,$ then $|T|^p-|S|^p\in(\mathcal{L}_{1,\infty})_0.$
\end{lemma}

The following lemma crucially uses Lemma \ref{main lemma} from the preceding section. Recall the lightened notation: the algebra $L_{\infty}(\partial\mathbb{D})$ is identified with its natural action on the Hilbert space $L_2(\partial\mathbb{D})$ by pointwise multiplication.

\begin{lemma}\label{final lemma} Let $f\in C(\mathbb{S}^1)$ be such that $[F,f]\in\mathcal{L}_{p,\infty}.$ Let $g\in{\rm SL}(2,\mathbb{C})$ be such that the function $u=\frac{g_{11}f+g_{12}}{g_{21}f+g_{22}}$ is well defined and bounded. We have
\begin{equation}\label{63 main}
|[F,u]|^p\in |[F,f]|^p\cdot |g_{21}f+g_{22}|^{-2p}+(\mathcal{L}_{1,\infty})_0.
\end{equation}
\end{lemma}
\begin{proof} Since $u$ is bounded, it follows that $f$ is separated from $-\frac{g_{22}}{g_{21}}\in\bar{\mathbb{C}}.$ Thus, $v=(g_{21}f+g_{22})^{-1}\in C(\mathbb{S}^1).$ If $g_{21}=0,$ then the assertion is trivial. Further, we assume that $g_{21}\neq0.$ Clearly, $u=\frac{g_{11}}{g_{21}}-\frac1{g_{21}}v.$ Thus,
$$[F,u]=-\frac1{g_{21}}[F,v]=\frac1{g_{21}}\cdot v[F,g_{21}f+g_{22}]v=v[F,f]v.$$
Therefore, we have
$$[F,u]=[F,f]v^2+[v,[F,f]]\cdot v.$$
Since $v\in C(\mathbb{S}^1),$ it follows from Theorem 8 (a) on p. 319 in \cite{Connes} that
$$[F,u]\in[F,f]v^2+(\mathcal{L}_{p,\infty})_0.$$
By Lemma \ref{bks lemma}, we have (everywhere in the proof below, $LHS$ means the left hand side of \eqref{63 main})
$$LHS\in\Big|[F,f]v^2\Big|^p+(\mathcal{L}_{1,\infty})_0.$$
Equivalently,
$$LHS\in\Big||[F,f]|v^2\Big|^p+(\mathcal{L}_{1,\infty})_0.$$

Since $v^2\in C(\mathbb{S}^1),$ it follows from Theorem 8 (a) (on p. 319 in \cite{Connes}) that
$$\Big[[F,f],v^2\Big]\in(\mathcal{L}_{p,\infty})_0.$$
By Lemma \ref{absolute value lemma}, we have
\begin{equation}\label{comm cond holds1}
\Big[|[F,f]|,v^2\Big]\in(\mathcal{L}_{p,\infty})_0.
\end{equation}
It follows from Lemma \ref{bks lemma} that
$$LHS\in\Big|v^2|[F,f]|\Big|^p+(\mathcal{L}_{1,\infty})_0.$$
Equivalently,
$$LHS\in\Big||v|^2|[F,f]|\Big|^p+(\mathcal{L}_{1,\infty})_0.$$

Since $|v|\in C(\mathbb{S}^1),$ it follows from Theorem 8 (a) (on p. 319 in \cite{Connes}) that
$$\Big[[F,f],|v|\Big]\in(\mathcal{L}_{p,\infty})_0.$$
By Lemma \ref{absolute value lemma}, we have
\begin{equation}\label{comm cond holds2}
\Big[|[F,f]|,|v|\Big]\in(\mathcal{L}_{p,\infty})_0.
\end{equation}

We have
$$|v|^2|[F,f]|=|v|\cdot|[F,f]|\cdot|v|-|v|\cdot[|[F,f]|,|v|].$$
Thus,
$$|v|^2|[F,f]|\in |v|\cdot|[F,f]|\cdot|v|+(\mathcal{L}_{p,\infty})_0.$$
It follows from Lemma \ref{bks lemma} that
$$\Big||v|^2\cdot|[F,f]|\Big|^p\in\Big||v|\cdot |[F,f]|\cdot |v|\Big|^p+(\mathcal{L}_{1,\infty})_0.$$
Thus,
$$LHS\in\Big||v|\cdot|[F,f]|\cdot |v|\Big|^p+(\mathcal{L}_{1,\infty})_0.$$

Set $A=|v|^2$ and $B=|[F,f]|.$ We have
$$LHS\in(A^{\frac12}BA^{\frac12})^p+(\mathcal{L}_{1,\infty})_0.$$
On the other hand, the equality \eqref{comm cond holds2} reads as follows: $[B,A^{\frac12}]\in(\mathcal{L}_{p,\infty})_0.$ It follows now from Lemma \ref{main lemma} that
$$LHS\in B^pA^p+(\mathcal{L}_{1,\infty})_0.$$
This is exactly \eqref{63 main} and the proof is complete.
\end{proof}

We also need the following auxiliary lemma. Page 314 in \cite{Connes} mentions a corresponding assertion for the Dirac operator on the line and the action of ${\rm SL}(2,\mathbb{R}).$ Those settings (and results) are unitarily equivalent.

\begin{lemma}\label{sl2r representation lemma} The mapping $h\to U_h,$ $h\in{\rm SU}(1,1),$ defined by the formula
$$(U_h\xi)(z)=\xi(\frac{\alpha z+\beta}{\bar{\beta}z+\bar{\alpha}})\frac1{\bar{\beta}z+\bar{\alpha}},\quad\xi\in L_2(\partial\mathbb{D}),\quad |z|=1,$$
where
$$h=
\begin{pmatrix}
\alpha&\beta\\
\bar{\beta}&\bar{\alpha}
\end{pmatrix},\quad |\alpha|^2-|\beta|^2=1,$$
is a unitary representation of the group ${\rm SU}(1,1)$ on the Hilbert space $L_2(\partial\mathbb{D})$ which commutes with $F.$
\end{lemma}
\begin{proof} The fact that $h\to U_h$ is a homomorphism is simple and we omit the proof.

First, we show this representation is unitary. Indeed, we have
$$\langle U_h\xi,U_h\xi\rangle=\frac1{2\pi}\int_{\partial\mathbb{D}}|\xi\circ h|^2(e^{it})\cdot \frac1{|\bar{\beta}e^{it}+\bar{\alpha}|^2}dt.$$
On the circle $\partial\mathbb{D},$ we have
$$h:e^{it}\to e^{is}\stackrel{def}{=}\frac{\alpha e^{it}+\beta}{\bar{\beta}e^{it}+\bar{\alpha}}.$$
Thus,
$$\frac{ds}{dt}=\frac1ie^{-is}\cdot \frac{d(e^{is})}{dt}=\frac{\bar{\beta}e^{it}+\bar{\alpha}}{i(\alpha e^{it}+\beta)}\cdot\frac1{(\bar{\beta}e^{it}+\bar{\alpha})^2}\cdot ie^{it}=\frac1{|\bar{\beta}e^{it}+\bar{\alpha}|^2}.$$
Thus,
$$\langle U_h\xi,U_h\xi\rangle=\frac1{2\pi}\int_{\partial\mathbb{D}}|\xi|^2(e^{is})ds=\langle h,h\rangle.$$
Thus, $U_h$ is indeed a unitary operator.

Let $P_+\stackrel{def}{=}E_D[0,\infty).$ Let $e_n(z)=z^n,$ $|z|=1,$ $n\in\mathbb{Z}.$ If $n\geq0,$ then
$$(U_he_n)(z)=\frac{(\alpha z+\beta)^n}{(\bar{\beta}z+\bar{\alpha})^{n+1}}=(\bar{\alpha})^{-n-1}(\alpha z+\beta)^n(1+\frac{\bar{\beta}}{\bar{\alpha}}z)^{-n-1}=$$
$$=(\bar{\alpha})^{-n-1}(\alpha z+\beta)^n\sum_{m=0}^{\infty}\binom{-n-1}{m}(\frac{\bar{\beta}}{\bar{\alpha}}z)^m.$$
The series converges uniformly on the unit circle $\mathbb{S}^1$ because $|\beta|<|\alpha|.$ The series contains only positive powers of $z$ and, therefore, $P_+U_he_n=U_he_n.$

It follows from the preceding paragraph that $P_+U_hP_+=U_hP_+.$ Taking the adjoint, we obtain $P_+U_h^{-1}P_+=P_+U_h^{-1}.$ Replacing $h$ with $h^{-1},$ we obtain $P_+U_hP_+=P_+U_h.$ Thus, $P_+U_h=U_hP_+.$ It follows that $U_h$ commutes with $F.$
\end{proof}

We are now ready to prove our main result.

\begin{proof}[Proof of Theorem \ref{main theorem} \eqref{maintb}] Consider the linear functional on $C(\Lambda(G))$ defined by the formula
$$f\to \varphi((f\circ Z)\cdot|[F,Z]|^p),\quad f\in C(\Lambda(G)),$$
where $\varphi$ is a continuous trace on $\mathcal{L}_{1,\infty}.$

It follows from boundedness of $\varphi$ and \eqref{115} that
$$|\varphi((f\circ Z)\cdot|[F,Z]|^p)|\leq\|\varphi\|_{\mathcal{L}_{1,\infty}^*}\|f\circ Z\|_{\infty}\|[F,Z]\|_{p,\infty}^p.$$
Thus, our functional is bounded and, by the Riesz Representation Theorem, it admits a representation of the form
\begin{equation}\label{funk rep}
\varphi((f\circ Z)\cdot|[F,Z]|^p)=\int_{\Lambda(G)}f(t)d\kappa(t),\quad f\in C(\Lambda(G)).
\end{equation}
Here, $\kappa$ is some Radon measure on $\Lambda(G).$

We claim that
\begin{equation}\label{kappa g invariant}
\int_{\Lambda(G)}(f\circ g^{-1})(t)d\kappa(t)=\int_{\Lambda(G)}f(t)|g'(t)|^pd\kappa(t),\quad f\in C(\Lambda(G)),\quad g\in G.
\end{equation}

To see this, let $\pi(G)\subset {\rm SU}(1,1)$ be the Fuchsian group as in the proof of part \eqref{mainta}. Let $h\to U_h$ be its unitary representation given in Lemma \ref{sl2r representation lemma}. It is immediate that
$$U_{\pi(g)}(\xi\cdot\eta)=(\xi\circ\pi(g))\cdot U_{\pi(g)}(\eta),\quad \xi\in L_{\infty}(\partial\mathbb{D}),\quad \eta\in L_2(\partial\mathbb{D}).$$
Thus,
$$U_{\pi(g)}ZU_{\pi(g)}^{-1}=Z\circ\pi(g)=g\circ Z,\quad (f\circ g^{-1}\circ Z)=U_{\pi(g)}^{-1}(f\circ Z)U_{\pi(g)}.$$

Since $U_{\pi(g)}$ commutes with $F,$ it follows from the preceding formula that
$$(f\circ g^{-1}\circ Z)|[F,Z]|^p=U_{\pi(g)}^{-1}(f\circ Z)U_{\pi(g)}|[F,Z]|^p=$$
$$=U_{\pi(g)}^{-1}(f\circ Z)|U_{\pi(g)}[F,Z]U_{\pi(g)}^{-1}|^p\cdot U_{\pi(g)}=U_{\pi(g)}^{-1}(f\circ Z)|[F,g\circ Z]|^p\cdot U_{\pi(g)}.$$
It follows from the unitary invariance of the trace $\varphi$ that
$$\varphi((f\circ g^{-1}\circ Z)\cdot|[F,Z]|^p)=\varphi((f\circ Z)\cdot|[F,g\circ Z]|^p).$$
By Lemma \ref{final lemma} with $f=Z,$ we have
$$|[F,g\circ Z]|^p\in |[F,Z]|^p\cdot (|g'|^p\circ Z)+(\mathcal{L}_{1,\infty})_0.$$
Since $\varphi$ vanishes on $(\mathcal{L}_{1,\infty})_0,$ it follows that
$$\varphi((f\circ g^{-1}\circ Z)\cdot|[F,Z]|^p)=\varphi((f\circ Z)\cdot|[F,Z]|^p\cdot (|g'|^p\circ Z))=$$
$$=\varphi(((f|g'|^p)\circ Z)\cdot|[F,Z]|^p)\stackrel{\eqref{funk rep}}{=}\int_{\Lambda(G)}f(t)|g'(t)|^pd\kappa(t).$$
This proves \eqref{kappa g invariant}. In other words, $\kappa$ is a geometric measure.

As explained in the (first few lines of the) proof of Lemma \ref{reduction to sullivan invert}, the group $G$ is geometrically finite. Theorem 1 in \cite{sullivan84} states that geometric (probability) measure on $\Lambda(G)$ is unique. Setting $c(G,\varphi)=\kappa(\Lambda(G))$ completes the proof.
\end{proof}

\section{Proof of Theorem \ref{main theorem} \eqref{maintc}}\label{app1}
Let us introduce the power semigroup as follows.
$$(P_sx)(t)=x(t^s),\quad t,s>0.$$
If $\omega$ is an extended limit which is  invariant under $P_s$ (we say that it is power invariant), then $\omega\circ\log$ is a state on $L_{\infty}(-\infty,\infty)$ which is dilation invariant. This state vanishes on every function whose support is bounded from above and is, therefore, identified with a dilation invariant extended limit on $L_{\infty}(0,\infty).$

In this section, we consider those extended limits which are dilation and power invariant. The following assertion is available as Theorem 8.6.8 in \cite{LSZ}. For convenience of the reader, we present a short proof here.

\begin{lemma}\label{dix to zeta} If $\omega$ is a dilation and power invariant extended limit, then
$${\rm Tr}_{\omega}(A)=(\omega\circ\log)\Big(t\to\frac1t{\rm Tr}(A^{1+\frac1t})\Big),\quad 0\leq A\in\mathcal{L}_{1,\infty}.$$
\end{lemma}
\begin{proof} We have
$$RHS=(\omega\circ\log)\Big(t\to\frac1t\sum_{n\geq0}(n+1)^{-1-\frac1t}\cdot ((n+1)\mu(n,A))^{1+\frac1t}\Big)).$$
We have
$$|(n+1)\mu(n,A)-((n+1)\mu(n,A))^{1+\frac1t}|\leq\sup\{|x-x^{1+\frac1t}|:\ 0\leq x\leq\|A\|_{1,\infty}\}=O(\frac1t)$$
as $t\to\infty.$ Therefore,
$$RHS=(\omega\circ\log)\Big(t\to\frac1t\sum_{n\geq0}(n+1)^{-\frac1t}\mu(n,A)\Big).$$
Set now
$$\beta=\sum_{n\geq0}\mu(n,A)\chi_{(\log(n+1),\infty)}.$$
Clearly, $\beta(u)=O(u)$ as $u\uparrow\infty.$ Using Theorem 8.6.7 in \cite{LSZ}, we infer
$$\omega(t\to\frac{\beta(t)}{t})=\omega(t\to\frac{h(t)}{t}),$$
where
$$h(t)=\int_0^{\infty}e^{-\frac{u}{t}}d\beta(u)=\sum_{n\geq0}(n+1)^{-\frac1t}\mu(n,A).$$
Thus,
$$RHS=(\omega\circ\log)\Big(t\to\frac1t\sum_{\log(n+1)<t}\mu(n,A)\Big)\stackrel{def}{=}\omega\Big(t\to\frac1{\log(t)}\sum_{n+1<t}\mu(n,A)\Big).$$
Since $A\in\mathcal{L}_{1,\infty},$ it follows that
$$\int_0^t\mu(s,A)ds=\sum_{n+1<t}\mu(n,A)+O(1).$$
This completes the proof.
\end{proof}

\begin{corollary} If $\omega$ is a dilation and power invariant extended limit, then $c(G,{\rm Tr}_{\omega})>0.$
\end{corollary}
\begin{proof} Let $T=|[F,Z]|^p.$ It follows from Lemma \ref{maintc core} that
$$\liminf_{s\to0}s{\rm Tr}(T^{1+s})>0.$$
Therefore,
$$(\omega\circ\log)\Big(t\to\frac1t{\rm Tr}(T^{1+\frac1t})\Big)>0.$$
The assertion follows now from Lemma \ref{dix to zeta}.
\end{proof}

\begin{remark}
The existence of a Dixmier trace $\varphi$ on $\mathcal{L}_{1,\infty}$ such that $\varphi(T)\neq0$ follows from the weaker estimate $\limsup_{s\to0}s{\rm Tr}(T^{1+s})>0$. Indeed, assume the contrary, that is $\varphi(T)=0$ for every Dixmier trace $\varphi.$ It follows from Theorem 9.3.1 in \cite{LSZ} that   
$$\lim_{s\to0}s{\rm Tr}(T^{1+s})=0,$$
which is not the case. 
Since $\varphi(T)=c(G,\varphi),$ the assertion follows.	
\end{remark}

\end{document}